\documentclass[12pt]{amsart}
\title[Tower of fully commutative elements of type $\tilde A$]{Tower of  fully commutative 
elements of type $\tilde A$  \\ 
 and  applications}
\author{Sadek AL HARBAT}
\address{} 
\email{sadikharbat@math.univ-paris-diderot.fr}

\usepackage{etex} 
\usepackage[latin1]{inputenc}	
\usepackage[T1]{fontenc}
\usepackage[english]{babel}
\usepackage{amsmath,amssymb}
\usepackage{wasysym}
\usepackage{stmaryrd}
\usepackage[table]{xcolor}
\usepackage{color}
\usepackage{dsfont}\let\mathbb\mathds
\usepackage[all]{xy}
\usepackage{graphicx}
\usepackage{mathenv}
\usepackage{lastpage}
\usepackage{fancyhdr}	
\usepackage{subfigure}
\usepackage{geometry}
\usepackage[version=3]{mhchem}
\usepackage[force,almostfull]{textcomp}
\usepackage{array}
\usepackage{lipsum}		
\usepackage{eso-pic}
\usepackage{multirow}
\usepackage{fancybox}
\usepackage{cite}
\usepackage{amsthm}
\usepackage{listings}
\usepackage{lscape}
\usepackage{multicol}		
\usepackage{setspace}
\usepackage{times}
\usepackage{eurosym}
\usepackage{latexsym}
\usepackage{ulem}
\usepackage{lmodern}
\usepackage{colortbl}
\usepackage{rotating}
\usepackage{type1cm}
\usepackage{lettrine}
\usepackage{ragged2e}
\usepackage{mathrsfs}
\usepackage{scrtime}
\usepackage{enumitem}

\usepackage{longtable}
\usepackage{url}
\usepackage{nomentbl}
\usepackage{nomencl}
\usepackage{hyperref}
\hypersetup{
pdfpagemode=UseOutlines,      
pdfstartview=Fit,             
pdffitwindow=true,            
pdfpagelayout=TwoColumnsRight,
pdftoolbar=true,              
pdfmenubar=true,              
bookmarksopen=false,          
bookmarksnumbered=true,       
colorlinks=true,              
pdfauthor={Ton nom},          
pdftitle={Titre PDF},         
pdfcreator=PDFLaTeX,          %
pdfproducer=PDFLaTeX,         %
linkcolor=blue,               
urlcolor=blue,                
anchorcolor=black,            
citecolor=blue,               
frenchlinks=true,             
pdfborder={0 0 0}             
}
\usepackage{pgf, tikz}
\usetikzlibrary{matrix,positioning}
\usetikzlibrary{arrows,decorations.markings}

\newtheorem{theorem}{Theorem}[section]

\newtheorem{definition}[theorem]{Definition}
\newtheorem{proposition}[theorem]{Proposition}
\newtheorem{lemma}[theorem]{Lemma}
\newtheorem{corollary}[theorem]{Corollary}
\newtheorem{example}[theorem]{Example}
\newtheorem{remark}[theorem]{Remark}

\newenvironment{demo}{\begin{proof}}{\end{proof}}

\addto\captionsfrenchb{}
\addto\captionsfrenchb{}

    \newlength{\myarrowsize} 
    \newlength{\myoldlinewidth}

    \pgfarrowsdeclare{myto}{myto}{
        \pgfsetdash{}{0pt} 
        \pgfsetbeveljoin 
        \pgfsetroundcap 
        \setlength{\myarrowsize}{0.5pt}
        \addtolength{\myarrowsize}{.5\pgflinewidth}
        \pgfarrowsleftextend{-4\myarrowsize-.5\pgflinewidth} 
        \pgfarrowsrightextend{.7\pgflinewidth}
    }{
        \setlength{\myarrowsize}{0.4pt} 
        \addtolength{\myarrowsize}{.3\pgflinewidth}  
        \setlength{\myoldlinewidth}{\pgflinewidth}
        \pgfsetroundjoin
        \pgfsetlinewidth{0.0001pt}
        \pgfpathmoveto{\pgfpoint{0.43\myarrowsize}{0}}
        \pgfpatharc{0}{70}{0.14\myarrowsize}
        \pgfpatharc{-80}{-169.5}{4\myarrowsize}
        \pgfpatharc{150}{189}{0.95\myarrowsize and 0.95\myarrowsize}
        \pgfpatharc{0}{-40}{15\myarrowsize}
        \pgfpathmoveto{\pgfpoint{0.43\myarrowsize}{0}}
        \pgfpatharc{0}{-70}{0.14\myarrowsize}
        \pgfpatharc{80}{169.5}{4\myarrowsize}
        \pgfpatharc{-150}{-189}{0.95\myarrowsize and 0.95\myarrowsize}
        \pgfpatharc{0}{40}{15\myarrowsize}
        \pgfpathclose
        \pgfsetstrokeopacity{0.25}
        \pgfusepathqfillstroke
    }

    \pgfarrowsdeclare{myonto}{myonto}{
        \pgfsetdash{}{0pt} 
        \pgfsetbeveljoin 
        \pgfsetroundcap 
        \setlength{\myarrowsize}{0.5pt}
        \addtolength{\myarrowsize}{.5\pgflinewidth}
        \pgfarrowsleftextend{-4\myarrowsize-.5\pgflinewidth} 
        \pgfarrowsrightextend{.7\pgflinewidth}
    }{
        \setlength{\myarrowsize}{0.4pt} 
        \addtolength{\myarrowsize}{.3\pgflinewidth}  
        \setlength{\myoldlinewidth}{\pgflinewidth}
        \pgfsetroundjoin
        \pgfsetlinewidth{0.0001pt}
        \pgfpathmoveto{\pgfpoint{0.43\myarrowsize}{0}}
        \pgfpatharc{0}{70}{0.14\myarrowsize}
        \pgfpatharc{-80}{-169.5}{4\myarrowsize}
        \pgfpatharc{150}{189}{0.95\myarrowsize and 0.95\myarrowsize}
        \pgfpatharc{0}{-40}{15\myarrowsize}
        \pgfpathmoveto{\pgfpoint{-7\myarrowsize}{0}}
				\pgfpatharc{0}{70}{0.14\myarrowsize}
        \pgfpatharc{-80}{-169.5}{4\myarrowsize}
        \pgfpatharc{150}{189}{0.95\myarrowsize and 0.95\myarrowsize}
        \pgfpatharc{0}{-40}{15\myarrowsize}
        \pgfpathmoveto{\pgfpoint{0.43\myarrowsize}{0}}
        \pgfpatharc{0}{-70}{0.14\myarrowsize}
        \pgfpatharc{80}{169.5}{4\myarrowsize}
        \pgfpatharc{-150}{-189}{0.95\myarrowsize and 0.95\myarrowsize}
        \pgfpatharc{0}{40}{15\myarrowsize}
			  \pgfpathmoveto{\pgfpoint{-7\myarrowsize}{0}}
        \pgfpatharc{0}{-70}{0.14\myarrowsize}
        \pgfpatharc{80}{169.5}{4\myarrowsize}
        \pgfpatharc{-150}{-189}{0.95\myarrowsize and 0.95\myarrowsize}
        \pgfpatharc{0}{40}{15\myarrowsize}
        \pgfpathclose
        \pgfsetstrokeopacity{0.25}
        \pgfusepathqfillstroke
        \pgfusepathqfillstroke
    }

 \pgfarrowsdeclare{myhook}{myhook}{
        \setlength{\myarrowsize}{0.6pt}
        \addtolength{\myarrowsize}{.5\pgflinewidth}
        \pgfarrowsleftextend{-4\myarrowsize-.5\pgflinewidth} 
        \pgfarrowsrightextend{.7\pgflinewidth}
    }{
        \setlength{\myarrowsize}{0.6pt} 
        \addtolength{\myarrowsize}{.5\pgflinewidth}  
        \pgfsetdash{}{+0pt}
        \pgfsetroundcap
        \pgfpathmoveto{\pgfqpoint{-2pt}{-6\pgflinewidth}}
        \pgfpathcurveto
            {\pgfqpoint{4\pgflinewidth}{-4.667\pgflinewidth}}
            {\pgfqpoint{4\pgflinewidth}{0pt}}
            {\pgfpointorigin}
        \pgfusepathqstroke
    }

		\pgfarrowsdeclare{my to}{my to}
{
  \pgfarrowsleftextend{-2\pgflinewidth}
  \pgfarrowsrightextend{\pgflinewidth}
}
{
  \pgfsetlinewidth{0.8\pgflinewidth}
  \pgfsetdash{}{0pt}
  \pgfsetroundcap
  \pgfsetroundjoin
  \pgfpathmoveto{\pgfpoint{-5.5\pgflinewidth}{7.5\pgflinewidth}}
  \pgfpathcurveto
  {\pgfpoint{-4.0\pgflinewidth}{0.1\pgflinewidth}}
  {\pgfpoint{0pt}{0.25\pgflinewidth}}
  {\pgfpoint{0.75\pgflinewidth}{0pt}}
  \pgfpathcurveto
  {\pgfpoint{0pt}{-0.25\pgflinewidth}}
  {\pgfpoint{-4.0\pgflinewidth}{-0.1\pgflinewidth}}
  {\pgfpoint{-5.5\pgflinewidth}{-7.5\pgflinewidth}}
  \pgfusepathqstroke
}

		\pgfarrowsdeclare{mytodashed}{mytodashed}
{
 \pgfarrowsleftextend{-2\pgflinewidth}
 \pgfarrowsrightextend{\pgflinewidth}
}
{
 \pgfsetlinewidth{0.8\pgflinewidth}
  \pgfsetdash{{0.8cm}{0.8cm}}{0cm}
\pgfsetroundcap
\pgfsetroundjoin
  \pgfpathmoveto{\pgfpoint{-5.5\pgflinewidth}{7.5\pgflinewidth}}
\pgfpathcurveto
{\pgfpoint{-4.0\pgflinewidth}{0.1\pgflinewidth}}
 {\pgfpoint{0pt}{0.25\pgflinewidth}}
{\pgfpoint{0.75\pgflinewidth}{0pt}}
\pgfpathcurveto
{\pgfpoint{0pt}{-0.25\pgflinewidth}}
  {\pgfpoint{-4.0\pgflinewidth}{-0.1\pgflinewidth}}
 {\pgfpoint{-5.5\pgflinewidth}{-7.5\pgflinewidth}}
\pgfusepathqstroke
}

\tikzstyle{vecArrow} = [thick, decoration={markings,mark=at position
   1 with {\arrow[semithick]{open triangle 60}}},
   double distance=1.4pt, shorten >= 5.5pt,
   preaction = {decorate},
   postaction = {draw,line width=1.4pt, white,shorten >= 4.5pt}]
\tikzstyle{innerWhite} = [semithick, white,line width=1.4pt, shorten >= 4.5pt]

	\makeatletter
	\newcommand\POSITION[3]{%
	\begingroup
	\@tempdim@x=0cm
	\@tempdim@y=\paperheight
	\advance\@tempdim@x#1
	\advance\@tempdim@y-#2
	\put(\LenToUnit{\@tempdim@x},\LenToUnit{\@tempdim@y}){#3}%
	\endgroup
	}
\makeatother

\addto\captionsenglish{}

\begin{document}

	\begin{abstract}
	Let $W^c(\tilde A_{n})$ be the set of  fully commutative elements in the affine Coxeter group 
		$W(\tilde A_{n})$ of type $\tilde{A}$. 
		We classify  the elements of $W^c(\tilde A_{n})$  and  give a normal form for them. 
We give a first application of this normal form to fully commutative affine braids.  We then 
use this normal form to define two injections from $W^c(\tilde A_{n-1})$ into $W^c(\tilde A_{n})$ and examine  their properties. We finally consider the tower of affine Temperley-Lieb algebras of type $\tilde{A }$ and use the injections above to prove the injectivity of this tower.    
	\end{abstract}

 	\maketitle
	

\section{Introduction}

Let $(W,S)$  be a Coxeter system. We say that w in  W is {\it fully commutative} if any reduced expression for $ w$
can be obtained from any other using only commutation relations among the members of the set $S$.
If $W$ is simply laced then the fully commutative elements of $W$ are those with no $sts$  factor in any 
reduced expression, where $t$ and $s$ are any non-commuting generators.\\

 In this paper we are interested with the affine Coxeter group of type $\tilde A$ which has an infinite set of fully commutative elements, as proved in  \cite{St96} 
 where Stembridge assigns to each fully commutative element $w$   a unique labeled partial order, called the heap of $ w$, whose linear extensions encode the reduced expressions for $w$. The notion of heap  was used frequently as a way to approach affine fully commutative elements, while other notions, for example abacus diagrams, were used in  \cite{HJ}. In this work we only use algebraic methods to deal with them, such as the affine length (see 
Definition~\ref{AL}).\\

In a given  Coxeter group the subset of fully commutative elements is indeed an interesting set  with many
remarkable properties, in particular  relating to Kazhdan-Lusztig polynomials (see for example \cite{BW}) hence relating to $\mu$-coefficients. Moreover, they play the most important role in the $M$-coefficients notion, see \cite{ Green2007}. Under certain conditions they are compatible with the classical Kazhdan-Lusztig cells, in the sense that   the set of fully commutative elements is a union of cells  \cite{Shi,GL}. There is also an intrinsic notion of cell coming from the structure of Temperley-Lieb algebra, those cells are classified in 
  \cite{FG}. \\

This paper is divided into two parts. The first part establishes a classification of
affine fully commutative elements  in type $\tilde A$: they are depicted by  a normal form given in 
Theorem~\ref{FC}. This form is similar to Stembridge's description of fully commutative elements in the Coxeter groups of finite type $ A, B, D$  \cite{St}, although a classification of fully commutative elements of type $A$ was given by Jones in \cite{Jones_1985} before even the official definition of fully commutative elements  in the 90's, see for example  \cite{Fan_1995,Graham}. \\

Classification is interesting in itself, nevertheless, since affine
fully commutative elements in type $\tilde A$ index a basis of the
affine Temperley-Lieb algebra  \cite{Fan_1996}, it is to have consequences on the
structure of the affine Temperley-Lieb algebra, on the  tower of affine Temperley-Lieb algebras  defined in
\cite{Sadek_Thesis}  and on the traces on this algebra. This is
precisely the point of the second part, which is divided into two
applications.\\

The first application is to give a general form for ``fully commutative braids'' as follows: we lift the fully commutative elements to elements having the same expression in the $\tilde{A}$-type braid group  $B(\tilde A_{n}) $, or: {\it fully commutative braids} (in this work we use the same symbols for the generators of the braid group and their images in the corresponding Coxeter group). Regarding $B(\tilde A_{n-1}) $
 as a subgroup of $B(\tilde A_{n}) $ by means of an injective homomorphism $R_n$, we  
  give in Theorem~\ref{2_5_2} a general form for these fully commutative braids in terms of elements of $B(\tilde A_{n-1}) $ and the lift $c_{n}$ of a certain Coxeter element  to $B(\tilde A_{n}) $.   The 
 tower of affine braid groups: 
		\begin{eqnarray}
			 \{1 \} {\longrightarrow} B(\tilde A_{1} )\stackrel{R_{1}}{\longrightarrow} \dotsb 
\stackrel{R_{n-1}}{\longrightarrow}   B(\tilde A_{n-1} ) \stackrel{R_{n}}{\longrightarrow}B(\tilde A_{n} ) \stackrel{R_{n+1}}{\longrightarrow}               \dotsb           \nonumber
		\end{eqnarray}	
gives rise to an analogous injective tower of the group algebras  $K[B(\tilde A_{n} )]$ over an integral domain $K$ 
of characteristic $0$. Let $q$ be an invertible element in $K$. The affine Temperley-Lieb algebra $\widehat{TL}_{n+1}(q)$ is a quotient of the braid group algebra $K[B(\tilde A_{n} )]$ and  we   get    (see Section~\ref{TL}) 
  a tower of affine Temperley-Lieb algebras: 
\begin{eqnarray}
			\widehat{TL}_{1}(q) \stackrel{R_{1}}{\longrightarrow}  \widehat{TL}_{2}(q)  \longrightarrow \dotsb  \stackrel{R_{n-1}}{\longrightarrow}
\widehat{TL}_{n}(q) \stackrel{R_{n}} {\longrightarrow}\widehat{TL}_{n+1}(q) \stackrel{R_{n+1}}\longrightarrow \dotsb  \nonumber 
		\end{eqnarray} 
The images by the quotient map  of  the fully commutative braids in $K[B(\tilde A_{n} )]$ make up a basis of 
$\widehat{TL}_{n+1}(q)$ and the 
  form for fully commutative braids given in Theorem~\ref{2_5_2} 
  is the key to the definition of {\it Markov elements} in  $\widehat{TL}_{n+1}(q)$,  and the key to 
proving  that any trace on the affine Temperley-Lieb algebra is  uniquely defined by its values on Markov elements    \cite[Theorem 4.6]{Sadek_2013_2}.  This in turn leads to the existence and
uniqueness of the affine Markov trace \cite{Sadek_2013_2} and on the other hand  is a step towards Green's conjectures (Property $B$)
\cite{Green2007}. \\

The second application is to prove the faithfulness of the arrows of the
tower of affine Temperley-Lieb algebras (Theorem \ref{F}). This was one of the most
interesting questions since defining this tower in \cite{Sadek_Thesis}. The faithfulness has consequences on the affine knot invariant defined in \cite{Sadek_Thesis}, and on the parabolic-like presentation defined in \cite{Sadek_Thesis} on the level of affine Hecke algebra and recently for  the affine Temperley-Lieb algebra.  \\

The paper is organized as follows:  

In Section 2, we give some general definitions, then we state and prove our  main result, Theorem~\ref{FC}: a normal form for   affine fully commutative elements in type $\tilde A$. This is the affine version of  
Theorem~\ref{1_2}.	 

In Section 3, we define the tower of affine braid groups and establish its faithfulness. 
We then define 
fully commutative braids  and, using our normal form and the fact that the lift $c_{n}$  of a certain Coxeter element  to $B(\tilde A_{n} ) $ acts as a Dynkin automorphism on $B(\tilde A_{n-1} ) $  (Lemma~\ref{Dynkin}), we find a general form for fully commutative braids in Theorem~\ref{2_5_2}.

In Section 4, we show   that  the set $W^c(\tilde A_{n-1})$  of fully commutative elements in the  Coxeter group with $n$ generators of type $\tilde A $  injects into $W^c(\tilde A_{n})$ in two different ways (Theorem~\ref{IJ}).   The existence of these two injections 
$I$ and $J$  depends totally on the normal form of Theorem~\ref{FC}. The intersection of their images is the image of  
$W^c( A_{n-1}) $ on which they coincide.   

In Section 5, we define the tower of affine  Temperley-Lieb algebras coming from the tower of affine braid groups, then we prove in Theorem~\ref{F}   the faithfulness of the arrows of this tower, using in a crucial way the injections $I$ and $J$ of the previous section.

\section{A normal form for affine fully commutative elements}\label{notations}

		 	Let $(W,S)$ be a Coxeter system with associated Dynkin Diagram $\Gamma$. For $s, t$ in $S$ we let $m_{st}$ be the order of $st$ in $W$. Let $w\in W$. We denote by $l(w)$ the length of a (any) reduced expression of $w$. We call 
			{\it support of $w$} and  denote by $Supp(w)$ the subset of $S$ consisting of all generators appearing in a (any) reduced expression of $w$.  We define $\mathscr{L} (w) $ to be the set of $s\in S$ such that $l(sw)<l(w)$, in other words  $s$ appears at the left edge of some reduced expression of $w$. We define $\mathscr{R}(w)$ similarly.\\

				We know that from a given reduced expression of $w$ we can arrive to any other reduced expression only by applying braid relations \cite[\S 1.5 Proposition 5]{Bourbaki_1981}. Among these relations there are commutation relations: those 
				that  correspond  to   generators $t$ and $s$ with $m_{st} = 2$.

	             \begin{definition}
			Elements for which one can pass from any reduced expression to any other one only by applying commutation relations are called {\rm fully commutative elements}. We denote  by $W^{c}$ the set of fully commutative elements in $W$. \\
		    \end{definition}

		The center of our interest in this work is fully commutative elements in $\tilde{A}$-type Coxeter groups. In this case fully commutative elements have some additional elegant properties, in particular:

		\begin{proposition}\label{HIT}{\rm \cite[Lemma 3.1]{HIT}}
			Let $(W,S)$ be  a Coxeter system such that $m_{st}$ is odd or 2 for any $s$, $t$ in $S$. Let $w\in W$. Then $w$ is fully commutative if and only if every $s$ in $Supp(w)$ occurs the same number of times in any reduced expression of $w$. \\
		\end{proposition}

			Consider the $A$-type Coxeter group with $n$ generators $W(A_{n})$, with the following Dynkin diagram:\\

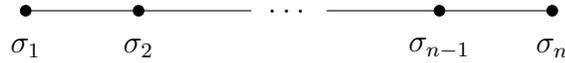
			\begin{figure}[ht]
				\centering
				 
				\begin{tikzpicture}

  \filldraw (0,0) circle (2pt);
  \node at (0,-0.5) {$\sigma_{1}$}; 
   
  \draw (0,0) -- (1.5, 0);

  \filldraw (1.5,0) circle (2pt);
  \node at (1.5,-0.5) {$\sigma_{2}$};

  \draw (1.5,0) -- (3, 0);

  \node at (3.5,0) {$\dots$};

  \draw (4,0) -- (5.5, 0);
  
  \filldraw (5.5,0) circle (2pt);
  \node at (5.5,-0.5) {$\sigma_{n-1}$};
 
  \draw (5.5,0) -- (7, 0);
  
  \filldraw (7,0) circle (2pt);
  \node at (7,-0.5) {$\sigma_{n}$};

               \end{tikzpicture}
			 
			\end{figure}					
			We let: 					
$$
\begin{aligned}
\lfloor i,j \rfloor &= \sigma_i \sigma_{i-1} \dots \sigma_j   \   \text{ for } n\ge i\ge j \ge 1    \  \text{ and } \    \lfloor 0,1 \rfloor = 1, 
\\
\lceil  i,j \rceil  &= \sigma_i \sigma_{i+1} \dots \sigma_j   \    \text{ for } 1\le i\le j \le n \    \text{ and }  \   \lceil  n+1, n \rceil  = 1, 
\\
\qquad  \quad    h(i,r) & =   \lfloor i,1 \rfloor \lceil  r,n \rceil 
  \quad   \text{ for } 0\le i < r \le n+1  \text{ and } (i,r)  \ne (0,1),  
  \end{aligned}
$$
 hence  
 $$
\begin{aligned}
 h(i,r)   &=  \sigma_i \sigma_{i-1} \dots \sigma_1 \sigma_r \sigma_{r+1} \dots \sigma_n 
  \quad   \text{ for } 1\le i < r \le n ,  \\ 
   h(0,r )   &=  \lceil  r,n \rceil  \text{ for } 2\le r \le n, \\  
  h(i, n+1) &= \lfloor i,1 \rfloor \text{ for } 1\le i \le n ,  \\  
h(0, n+1) &= 1 .   
\end{aligned}
$$

 \medskip
Considering right classes of $W(A_{n-1})$ in $W(A_{n}) $, Stembridge  has described   canonical reduced words for   elements of 
$W(A_n)$, namely:
$$
  \lceil  m_1, n_1 \rceil     \lceil  m_2, n_2 \rceil \dots   \lceil  m_r, n_r \rceil  
$$
where $n\ge n_1 > \dots > n_r \ge 1$ and $n_i \ge m_i \ge 1$ \cite[p.1288]{St}. He also proved    \cite[Corollary 5.8]{St}  that fully commutative elements are those for which the canonical reduced word satisfies   $m_1 > \dots > m_r$. The set of fully commutative elements is stable under the inverse map; taking inverses we get:

\begin{theorem}\label{1_2}{\rm \cite[Corollary 5.8]{St}}
    $W^c(A_n)$ is the set of elements of the   form: 
 \begin{equation}\label{eq:Stembridge}
 \lfloor l_1, g_1 \rfloor \lfloor l_2, g_2 \rfloor \dots  \lfloor l_s, g_s \rfloor  , 
 \text{ with } 
 \left\{ \begin{matrix}
 1\le l_1 < \dots < l_s \le n ,  \cr   1\le g_1 < \dots < g_s \le n,   \cr
 l_t \ge g_t   \text{ for }  1\le t \le s.  
 \end{matrix}\right.
 \end{equation}
\end{theorem} 

\medskip

Inspecting the inequalities  above, we see that  the only term in expression (\ref{eq:Stembridge}) 
in which  $ \sigma_{n} $ 
  can   occur  is the $s$-th term. If $ \sigma_{n} $ does occur, then 
$l_s$ must be equal to $n$ and, whether or not $g_s$ is equal to  $n$, $\sigma_{n} $  occurs only once.   Similarly,  if $\sigma_{1} $ does occur in expression (1), 
  then $g_1=1$ and $\sigma_{1} =\sigma_{g_{1}}$  will   appear only once. \\ 
			
 	\begin{definition}
				An element $u$ in $W^{c}(A_{n})$ is called 
				{\rm extremal}  if   both $ \sigma_{n} $ and $ \sigma_{1} $ belong to $Supp(u)$. \\ 
\end{definition}

\begin{lemma}
 An extremal  element different from  $ \lfloor n,1 \rfloor $  can be written as 
$$
h(i,r) \; x  \ \text{ with } \   1\le i < r \le n   \ \text{ and } \ \text{Supp} (x) \subseteq 
\{ \sigma_2, \dots , \sigma_{n-1} \}. 
$$
\end{lemma}       			
 \begin{proof} An extremal  element  $u$ has a reduced expression of the form (\ref{eq:Stembridge}) 
				above with $g_1=1$ and $l_s=n$. 
If $s=1$ we have  $ u= \lfloor n,1 \rfloor =  h(n, n+1)$, the only extremal element for which 
the leftmost term in the reduced expression (\ref{eq:Stembridge})  is $\sigma_n$. 

\medskip 

Assume $u \ne  \lfloor n,1 \rfloor$. 
For $n=1$ we have  $\sigma_1=\lfloor 1,1 \rfloor$. For $n=2$ the element $\sigma_1 \sigma_2$ is the only extremal element different from $\lfloor 2,1 \rfloor$  
 and it is equal to $h(1,2)$. 
 Assume  now $n\ge 3$. 

The rightmost term in (\ref{eq:Stembridge})  is 
 $\lfloor n, g_s \rfloor $ with $ g_s > 1 $, so the generators on the right of $\sigma_n$, if any, belong to $\{ \sigma_2, \dots , \sigma_{n-1} \}$. 
 The generator  $\sigma_n$  commutes with any element in 
 $\{ \sigma_2, \dots , \sigma_{n-2} \}$, so,  using the commutation relation
 $\sigma_i \sigma_n = \sigma_n \sigma_i$ for $2 \le i \le n-2$, 
  we can repeatedly 
  push  $\sigma_n$ to the left in  expression (\ref{eq:Stembridge})  above until the element on the left of $\sigma_n$ is  either $\sigma_{n-1}$ or $\sigma_1$. In this process all generators  $\sigma_i$  that are pushed to the right of $\sigma_n$ belong to $\{ \sigma_2, \dots , \sigma_{n-1} \}$. 
  
  If we arrive at a subexpression $a=\sigma_{n-1}\sigma_n$, which happens if and only if $l_{s-1}= n-1$,   
  then again it 
commutes with any element in 
 $\{ \sigma_2, \dots , \sigma_{n-3} \}$ so we can push $a$ to the left until  
 the element on the left of $a$ is  either $\sigma_{n-2}$ or $\sigma_1$. We continue in this way as long as  $l_{s-t}= n-t$    until we reach $\sigma_1$, 
and  obtain the final  expression   
 $ \lfloor l_1,  1 \rfloor  \lceil  n-k, n  \rceil x$, with 
 $  k = \max\{t  \  | \  0\le t < n-1  \text{ and } l_{s-t}= n-t \}$,  as announced. 
  \end{proof}	
  
  \medskip

			Now let $W(\tilde A_{n} ) $ be the affine Coxeter group of $\tilde{A}$-type with $n+1$ generators, with the following Dynkin diagram: 

\medskip
		 
			
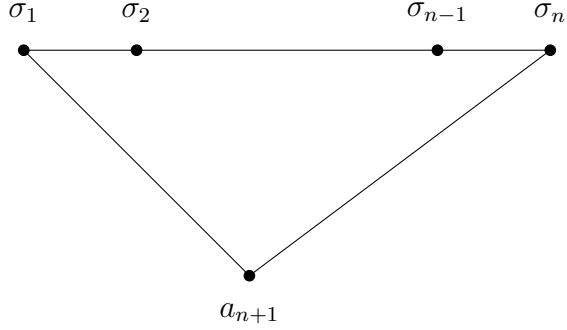
\begin{figure}[ht]
				\centering
			 
				\begin{tikzpicture}

 \node at (0,0.5) {$\sigma_{1}$}; 
  \filldraw (0,0) circle (2pt);
   
  \draw (0,0) -- (1.5, 0);
  
  \node at (1.5,0.5) {$\sigma_{2}$};
  \filldraw (1.5,0) circle (2pt);

  \draw (1.5,0) -- (5.5, 0);

  \node at (5.5,0.5) {$\sigma_{n-1}$};
  \filldraw (5.5,0) circle (2pt);
 
  \draw (5.5,0) -- (7, 0);
  
  \node at (7,0.5) {$\sigma_{n}$};
  \filldraw (7,0) circle (2pt);

  \draw (7,0) -- (3, -3);
  
  \filldraw (3, -3) circle (2pt);
  \node at (3, -3.5) {$a_{n+1}$};

  \draw (3, -3) -- (0, 0);
  
               \end{tikzpicture}
		 
			\end{figure}

Our notation encapsulates the fact that we view $W( {A_{n}}) $ as the parabolic subgroup of 
$W(\tilde A_{n} ) $ generated by $\sigma_1$, \dots, $\sigma_n$. 
Recalling Proposition~\ref{HIT} we make the following definition. 

				\begin{definition} \label{AL}
				We define the {\rm affine length} of $u$ in $W^{c}(\tilde A_{n})$ to be the number of times $a_{n+1}$ occurs in a (any) reduced expression of $u$. We denote  it by $L(u)$.   
			\end{definition}

\begin{lemma}\label{lemmafull}
Let $w$ be  a fully commutative element in $W(\tilde A_n)$ with $L(w) =m \ge 2$. 
Fix  a reduced expression of $w$ as follows: 
$$
w = u_1 a_{n+1} u_2 a_{n+1} \dots u_m a_{n+1} u_{m+1}  
$$ 
with $u_i$, for $1\le i \le m+1$,  a reduced expression of a fully commutative element in $W^c(A_n)$. 
Then $u_2$, \dots, $ u_m$ are extremal elements and there is a reduced expression of $w$ of the  form: 
 \begin{equation}\label{eq:forme1}
w =  h(i_1, r_1) a_{n+1} h(i_2, r_2) a_{n+1} \dots h(i_m, r_m) a_{n+1} \;  v_{m+1}
 \end{equation} 
where $ v_{m+1} \in W^c(A_n)$,  $0\le i_1 < r_1 \le n +1$, $(i_1, r_1)\ne (0,1)$  and,  for $2\le t\le m$,  we have either  $1\le i_t < r_t \le n $ or 
$(i_t,r_t )=(n,  n+1)$.
 \end{lemma} 
 \begin{proof}
 We first remark that in the same manner as in the previous lemma, we can write any 
 fully commutative element in  $W( A_n)$ as 
 $
h(i,r) \; x  $  with   $0 \le i < r \le n +1$, $(i, r)\ne (0,1)$    and $ \text{Supp} (x) \subseteq 
\{ \sigma_2, \dots , \sigma_{n-1} \}$, in particular  $x$ commutes with $a_{n+1}$. 
Writing in this way $u_1= h(i_1, r_1) x_1$, we can 
push $x_1$ to the right of $a_{n+1}$, obtaining a new term $u_2$ that we in turn  
write $  h(i_2, r_2) x_2$ with $x_2$ commuting to $a_{n+1}$. 
Proceeding from left to right, we   obtain formally form (\ref{eq:forme1}). 

It remains to show that the elements $u_i$, $2\le i \le m$,  are extremal. Indeed,  if 
the support of 
some such $u_i$  was contained in $\{ \sigma_2, \dots , \sigma_{n-1} \}$,
this $u_i$ would commute with $a_{n+1}$  and we would get a reduced expression 
containing $a_{n+1} a_{n+1}$, a contradiction.  Now if some such $u_i$   contained only one of $  \sigma_1,  \sigma_{n }  $, then, using the commutation relations,  $a_{n+1} u_i a_{n+1}$ could be written 
$\dots a_{n+1} \sigma_1 a_{n+1} \dots$ or $\dots a_{n+1} \sigma_n a_{n+1} \dots$, hence would contain a braid, which is impossible in a reduced expression for a fully commutative element.  \end{proof}

\begin{lemma}\label{fullandsigma} 
Let $1\le l \le n$ and  $0\le i   < r  \le n +1$, $(i, r)\ne (0,1)$. Then $w=h(i,r) \; a_{n+1} \; \sigma_l$ is a reduced fully commutative element if and only if one of the following holds: 
\begin{enumerate}
\item $l=r-1 = i $;   
\item $i < l  < r$. 
\end{enumerate}
\end{lemma} 
\begin{proof}
Assume first that $r \le n$ and 
write $w=   \lfloor i,1 \rfloor \lceil  r,n \rceil a_{n+1} \; \sigma_l$. 
Using commutation relations we  push $\sigma_l$ to the left as long as it commutes with its left neighbour.
\begin{itemize}
\item  If 
  $l \ge r$  we will arrive 
at the braid $\sigma_l \sigma_{l+1}\sigma_l$ if $l<n$, 
$\sigma_n a_{n+1}\sigma_n$ if $l=n$: $w$ is not fully commutative. 
\item  Assume    $l < r$. If $i=0$, indeed $w$ is reduced fully commutative. 
We proceed with $i \ge 1$.  
\begin{itemize}
\item  If $l < r-1$ and $l \le i$, again we   get a braid 
  $\sigma_{l }\sigma_{l-1}\sigma_l  $ 
  by pushing $\sigma_l$   leftmost hence $w$ is not fully commutative. 
  \item 
   If $i < l$ indeed $w$ is reduced fully commutative. 
 \item  If $r-1 > 1$, then $\sigma_{r-1}$ commutes with $a_{n+1} $  so for $l=r-1$ we  get    $ \sigma_i \dots \sigma_1 \sigma_r \sigma_{r-1} \sigma_{r+1} \cdots \sigma_n a_{n+1}$ which is  reduced fully commutative. 
 \item Finally if  $r-1=1=l$, then  $\sigma_1$ cannot get past $a_{n+1} $  on the left and  again we have a reduced fully commutative element. 
\end{itemize}
\end{itemize} 
If $r=n+1$ and  either $l < i$ or   $l=i <n$,   the same process produces  a braid $\sigma_{l }\sigma_{l-1}\sigma_l  $ 
or (if $l=1$) $\sigma_{1 }a_{n+1}\sigma_1  $,  hence $w$  is not fully commutative, while, for 
  $i<l$ or $l=i =n$, $w$ is reduced fully commutative. 
\end{proof}

\begin{lemma}\label{twofull}
 Let $h(i,r)$ and $h(i',r')$ be extremal elements different from  $ \lfloor n,1 \rfloor $. Then 
$$ w= h(i,r) \;a_{n+1} \; h(i',r') \;    $$ is a reduced fully commutative element if and only if one of the following holds: 
\begin{enumerate}
\item $i < i' < r' < r$; 
\item $i \le  i'$ and $r'=r=i'+1$.   
\end{enumerate}
\end{lemma} 
\begin{proof} We have by assumption 
  $1\le i < r \le n$ and  
	 $ 1\le i' < r' \le n$. We write 
	 $$w=     \lfloor i,1 \rfloor \lceil  r,n \rceil 
	   \;a_{n+1} \; \lfloor i',1 \rfloor \lceil  r',n \rceil .$$ 
	   
	 Assume $w$ is reduced fully commutative. 
From the previous lemma we must have 
 $i'=r-1 =i $ or $i < i' < r$. We know examine $r'$, after  
   noticing  that if $r' > i'+1$, then 
	$ \lfloor i',1 \rfloor  \sigma_{r'} =  \sigma_{r'}  \lfloor i',1 \rfloor $ 
	so   Lemma~\ref{fullandsigma}  imposes $r' = r-1=i $ or $i < r' <r$. \\

	If $i'=r-1  =i$,  	then 
	$r' > i'+1=r$ is impossible  by the previous remark, while  
  $r'=i'+1=r$ produces a reduced fully commutative $w$. 	\\

 If $i < i' < r$ then 
\begin{itemize}
\item  if $r' > i'+1$, the previous remark gives  $r'=r-1$ or $i < r' < r$, whence $i < i' < r' < r$, and under this condition $w$  is  reduced fully commutative; 
\item  if $r'=i'+1\le  r$, we can write  
 $$w=     \lfloor i,1 \rfloor \lceil  r,n \rceil 
	   \;a_{n+1} \;    \sigma_{i'} \sigma_{i'+1} \lfloor i'-1,1 \rfloor \lceil  i'+2,n \rceil  .$$   
We claim that no braid relation  involving $ \sigma_{i'}$ or $ \sigma_{i'+1}$ can occur.
On the right of the product $  \sigma_{i'}\sigma_{i'+1}$ in the expression just above, this  is clear.  This same product  
$  \sigma_{i'}\sigma_{i'+1}$  can be pushed to its left as long as it commutes with its left neighbour. 
	   If $r'=i'+1<   r-1$ we arrive at  
	   $$w=     \lfloor i,1 \rfloor  \sigma_{i'} \sigma_{i'+1}  \lceil  r,n \rceil 
	   \;a_{n+1} \;   \lfloor i'-1,1 \rfloor \lceil  i'+2,n \rceil   $$  
	where we see that our claim holds. 
 If $r'=i'+1=    r-1$	 we arrive at  
	   $$w=     \lfloor i,1 \rfloor  \sigma_{r-2} \sigma_{r } \sigma_{r-1}  \lceil  r+1,n \rceil 
	   \;a_{n+1} \;   \lfloor i'-1,1 \rfloor \lceil  i'+2,n \rceil   $$   
	   and our claim holds again. 
Finally if $r'=i'+1=    r $ 	   we arrive at  
	   $$
\begin{aligned}
w&=     \lfloor i,1 \rfloor    \sigma_{r } \sigma_{r-1} \sigma_{r+1} \sigma_{r }  \lceil  r+2,n \rceil 
	   \;a_{n+1} \;   \lfloor i'-1,1 \rfloor \lceil  i'+2,n \rceil   
\quad \text{if } r<n, 
\\
w&=     \lfloor i,1 \rfloor    \sigma_{n } 
	   \;a_{n+1} \;    \sigma_{n-1 }   \sigma_{n }  \lfloor n-2,1 \rfloor   
\quad \text{if } r=n,  \end{aligned} 
$$  
	   and our claim is proved, hence $w$ is   reduced fully commutative. 
\end{itemize}
\end{proof}

\begin{lemma}\label{right} 
Let $w \in W^c(\tilde A_n)$ with $L(w) = m \ge 2$. Write $w$  as in  {\rm (\ref{eq:forme1})}    and assume that  $h(i_t, r_t) \ne  \lfloor n,1 \rfloor  $ for $2\le t \le m$. 
There exist nonnegative integers $p$ and $k$ satisfying 
$p+k=m$ and an integer $j \in \{1,  \dots , n-1\} $ such that 
  $w$ has  the following form:
   \begin{equation}\label{forme2}
\begin{aligned}
\text{if } p=0 : w &=    (h(j,j+1) \; a_{n+1} )^k  
\;    w_r ,     \\  
\text{if } p>0 :
w &=   h(i_1, r_1) \; a_{n+1} \dots h(i_p, r_p) \;  a_{n+1}    (h(j,j+1) \; a_{n+1} )^k  
\;    w_r ,    \end{aligned}
 \end{equation} 
with  $w_r \in W^c(A_n)$   and, if $p > 0$: 

\begin{itemize}
\item 
 $ \  
1\le i_2 < \dots < i_p < r_p < \dots < r_2 \le n   \  $ and  $  \  r_p-i_p \ge 2$, 

\item $\;$ if $k > 0$:  either $i_p < j < j+1 < r_p$ or  $j+1 =  r_p $.  \\ 
\end{itemize}

The  element $w_r$ can be described as follows: 
\begin{itemize}
\item if $k> 0$:  for some $z  \in \{0,  \dots , n\} $ such that $j+z-1\le n$, we have 

\begin{itemize}
\item  if $z=0$ :  $w_r =1$; 

\item   if $z\ge 1$ :  
$w_r=  \lfloor j, d_1 \rfloor \lfloor j+1, d_2 \rfloor \dots  \lfloor j+z-1, d_{z} \rfloor $ 

with $1\le d_1 < \dots <d_{z} \le n$  and 
$j+c\ge d_{c+1}$ for $0\le c\le z-1$; \\ 
\end{itemize}
\item if $k=0$ (hence $p>0$): for some $t  \in \{0,  \dots , n\} $ we have 
\begin{itemize}
\item if $t=0$ :  $w_r =1$; 

\item if $t\ge 1$:   $ w_r= 
 \lfloor l_1, g_1 \rfloor \lfloor l_2, g_2 \rfloor \dots  \lfloor l_t, g_t \rfloor  
$
with:  
\begin{itemize}
\item  $1\le l_1 < \dots < l_t \le n$, 
\item $1\le g_1 < \dots < g_t \le n$, 
\item $l_i \ge g_i$ for $1\le i\le t$,  
\item $i_p < l_1 < r_p$,  
\item for any $i$, $2 \le i \le t$, such that $l_{i} > l_{i-1}+ 1$ we have 
$l_i < r_p$. 
\end{itemize}
\end{itemize}
\end{itemize}
 \end{lemma} 
\begin{proof}
 We apply Lemma~\ref{twofull}  repeatedly from left to right, i.e. letting $t$ increase in 
 form  (\ref{eq:forme1})  (Lemma~\ref{lemmafull}). Case  (1)   forces the inequality  $i < i' < r' < r$, it cannot happen more than $[n/2]$ times.   Case (2) can only be followed by $h(i',i'+1)$ again. We thus get (\ref {forme2}). 
 
To determine $w_r$ we use form (\ref{eq:Stembridge}) (Theorem~\ref{1_2}) and we repeatedly apply Lemma~\ref {fullandsigma}, remembering that if $l_{i} > l_{i-1}+ 1$, there is a reduced expression for $w_r$ that begins with $\sigma_{l_i}$. 
If $k> 0$  any reduced expression for $w_r$ has to begin with $\sigma_j$ on the left, we thus obtain a simple condition.  
  \end{proof}
  
  \medskip 

We can now start the classification. Let  
$w \in W(\tilde A_n)$ with $L(w) = m \ge 1$, written  as in (\ref{eq:forme1}) in Lemma~\ref{lemmafull}. 
We discuss according to the first term  $h(i_1, r_1)$ which can have one, and only one, of the following forms: 
\begin{enumerate}
\item  $i_1=n, r_1=n+1$  (extremal, equal to  $ \lfloor n,1 \rfloor $);
\item   $1\le i_1< r_1 \le n$ (extremal, different from  $ \lfloor n,1 \rfloor $);
\item    $i_1 = 0$ and $2 \le r_1 \le n$ ($\sigma_n$ appears but not $\sigma_1$); 
\item  $1 \le i_1 \le n-1$ and $r_1= n+1$ ($\sigma_1$ appears but not $\sigma_n$); 
\item   $i_1=0$ and $r_1=n+1$.  
\end{enumerate}
We now examine each case. 
\begin{enumerate}
\item 
By Lemma~\ref{fullandsigma}, 
after $h(n, n+1)=  \sigma_n \sigma_{n-1} \dots \sigma_1$ there is only one choice  
for $h(i_2, r_2)$, namely  $h(n, n+1)$ itself, and this repeats until we reach the rightmost  term. This term must be $1$ or some  $\lfloor n, i \rfloor $.  So $w$ has the form: 
$$  (h(n,n+1) \; a_{n+1} )^k  
\;    w_r  \text{ with  }  w_r=1\text{ or  }w_r= \lfloor n, i \rfloor .$$
\item
By Lemma~\ref{fullandsigma} all elements $h(i_t, r_t)$ must also be extremal and different from  $ \lfloor n,1 \rfloor $, hence the previous discussion applies and we get the forms in 
Lemma~\ref{right}, either with $p=0$
 (if $r_1= i_1 +1$), or with $p>0$. We only have to extend  the condition of the lemma to $(i_1, r_1)$, that is: 
 
$ \  
1\le i_1 <\dots < i_p < r_p < \dots < r_1 \le n  \  $,  $  \  r_p-i_p \ge 2$ and, 
  if $k > 0$,   either $i_p < j < j+1 < r_p$ or  $j+1 =  r_p $.
 
\item Assume first   $m\ge 2$. 
 By Lemma~\ref{fullandsigma} again we must have $i_2 < r_1 $ hence, as 
  before, all elements $h(i_t, r_t)$ in form   (\ref{eq:forme1}) must   be extremal and different from  $ \lfloor n,1 \rfloor $. Furthermore, if $r_2-i_2>1$, 
then $\sigma_{r_2} $ can be pushed to the left of  $h(i_2, r_2)$ so 
Lemma~\ref{fullandsigma} implies   $r_2 < r_1$.   
We obtain the possible forms from Lemma~\ref{right} 
with $p>0$, with the condition: 

$ \  
0= i_1 <  i_2 < \dots < i_p < r_p < \dots< r_2  < r_1 \le n  \  $,  $  \  r_p-i_p \ge 2$ 

and, 
  if $k > 0$,   either $i_p < j < j+1 < r_p$ or  $j+1 =  r_p $.

 If $m=1$, we have to describe the rightmost term $v_2$ in expression  (\ref{eq:forme1}). 
 Writing $v_2$ in   form (\ref{eq:Stembridge}) as in the proof of Lemma \ref{right}, 
 we find exactly the same expression as in the case $k=0$ for $w_r$ in this lemma, with $p=1$.   
 
\item If $m\ge 2$, 
   Lemma \ref{fullandsigma} gives $i_2 > i_1$. 
If $i_2=n$, from   the same lemma we obtain 
$$
w= h(i_1, n+1) a_{n+1} (h(n, n+1) a_{n+1} )^k    w_r 
$$
with $k>0$ and $w_r =1$ or $w_r=  \lfloor n, i \rfloor $ for some $i$, $1\le i \le n$. 

If $i_2 < n $ we are back to the case of extremal elements different from  $ \lfloor n,1 \rfloor $
and we obtain the possible forms from Lemma ~\ref{right} 
with $p>0$, with the condition: 

$ \  
1\le  i_1 <  i_2 < \dots < i_p < r_p < \dots < r_2 <  r_1= n  +1 \  $,  $  \  r_p-i_p \ge 2$ and, 
  if $k > 0$,   either $i_p < j < j+1 < r_p$ or  $j+1 =  r_p $. 
  
  If $m=1$, Lemma~\ref{right}  provides, for the rightmost term, the same expression as in the case $k=0$ for $w_r$ in this lemma, with $p=1$.

\item Here  if $m\ge 2$, in case $i_2 <n$, then in the notation of Lemma~\ref{right} we have $h(i_1, r_1)= 1$,  and in case $i_2=n$, we have a form similar to case (1),  
 namely 
$$ a_{n+1} (h(n,n+1) \; a_{n+1} )^k  
\;    w_r  \text{ with  }  w_r=1\text{ or  }w_r= \lfloor n, i \rfloor .$$ 
  If $m=1$ the rightmost term can be any fully commutative element $w_r$ in 
  $W  (A_n)$. Actually the description of $w_r$ given in   Lemma~\ref{right}, with $k=0$, $p=1$, 
  $i_1=0$ and $r_1=n+1$,  applies here as well.  \\
    
\end{enumerate}

We subsume this discussion in the following theorem: \\

\begin{theorem}\label{FC} 
Let $w \in W^c(\tilde A_n)$ with $L(w)   \ge 1$.  
Then $w$ can be written in a unique way as a reduced word of    the following form, for nonnegative integers $p$ and $k$ and for 
$j \in \{1,  \dots , n\} $:  
 \begin{equation}\label{formefinale}
\begin{aligned}
\text{if } p=0 : w &=    (h(j,j+1) \; a_{n+1} )^k  
\;    w_r  \ (\text{with } k>0),  \\  
\text{if } p>0 :
w &=   h(i_1, r_1) \; a_{n+1} \dots h(i_p, r_p) \;  a_{n+1}    (h(j,j+1) \; a_{n+1} )^k  
\;    w_r ,    \end{aligned}
 \end{equation} 
with  $w_r \in W^c(A_n)$   and, if $p > 0$: 

\begin{itemize}
\item 
 $ \  
0\le i_1< \dots < i_p < r_p < \dots  < r_1 \le n +1   \  $ and  $  \  r_p-i_p \ge 2$, 

\item if $k > 0$:  either $i_p < j < j+1 < r_p$ or  $j+1 =  r_p $.  
\end{itemize}
In particular  we have  $p \le [\frac{n+1}{2}] $. \\ 

The   element $w_r$ can be described as follows: 
\begin{itemize}
\item If $k> 0$:  for some $z  \in \{0,  \dots , n\} $ such that $j+z-1\le n$, we have 

\begin{itemize}
\item  if $z=0$ :  $w_r =1$; 

\item   if $z\ge 1$ :  
$w_r=  \lfloor j, d_1 \rfloor \lfloor j+1, d_2 \rfloor \dots  \lfloor j+z-1, d_{z} \rfloor $ 

with $1\le d_1 < \dots <d_{z} \le n$  and 
$j+c\ge d_{c+1}$ for $0\le c\le z-1$. \\ 
\end{itemize}
\item If $k=0$ (hence $p>0$): for some $t  \in \{0,  \dots , n\} $ we have 
\begin{itemize}
\item if $t=0$ :  $w_r =1$; 

\item if $t\ge 1$:   $ w_r= 
 \lfloor l_1, g_1 \rfloor \lfloor l_2, g_2 \rfloor \dots  \lfloor l_t, g_t \rfloor  
$
with:  
\begin{itemize}
\item  $1\le l_1 < \dots < l_t \le n$, 
\item $1\le g_1 < \dots < g_t \le n$, 
\item $l_i \ge g_i$ for $1\le i\le t$,  
\item $i_p < l_1 < r_p$; 
\item for any $i$, $2 \le i \le t$, such that $l_{i} > l_{i-1}+ 1$ we have 
$l_i < r_p$. \\
\end{itemize}
\end{itemize}
\end{itemize}
Conversely, any word written as  above is reduced and fully commutative. \\  
\end{theorem}

\begin{example}				Let $w$ be in $W^{c}(\tilde A_{2} )$. Then there exists $k \ge 0$  such that  $w$ has one and only one of the following forms:\\

					\begin{figure}[ht]
				\centering
			 
\begin{tikzpicture} 
				\begin{scope}[xscale = 1]

  \node[left =-13pt] at (-1.5,1)  {$1$};
  \node[left =-13pt] at (-1.5,0)  {$a_{3}$};
  \node[left =-13pt] at (-1.5,-1) {$\sigma_{1} a_{3}$};

	\draw (-1,1)  -- (0,0);
	\draw (-1,0)  -- (0,0);
	\draw (-1,-1) -- (0,0);
	
  \node at (1,0)  {$(\sigma_{2}\sigma_{1}a_{3})^{k}$};

	\draw (2,0)  -- (3,-1);
	\draw (2,0)  -- (3,0);
	\draw (2,0)  -- (3,1);

  \node[right =-13pt] at (3.5,1)  {$1$};
  \node[right =-13pt] at (3.5,0)  {$\sigma_{2}$};
  \node[right =-13pt] at (3.5,-1) {$\sigma_{2} \sigma_{1}$};

\end{scope}
\begin{scope}[xscale = 1, xshift = 8cm]

  \node[left =-13pt] at (-1.5,1)  {$1$};
  \node[left =-13pt] at (-1.5,0)  {$a_{3}$};
  \node[left =-13pt] at (-1.5,-1) {$\sigma_{2} a_{3}$};

	\draw (-1,1)  -- (0,0);
	\draw (-1,0)  -- (0,0);
	\draw (-1,-1) -- (0,0);
	
  \node at (1,0)  {$(\sigma_{1}\sigma_{2}a_{3})^{k}$};

	\draw (2,0)  -- (3,-1);
	\draw (2,0)  -- (3,0);
	\draw (2,0)  -- (3,1);

  \node[right =-13pt] at (3.5,1)  {$1$};
  \node[right =-13pt] at (3.5,0)  {$\sigma_{1}$};
  \node[right =-13pt] at (3.5,-1) {$\sigma_{1} \sigma_{2}$};

\end{scope}
   \end{tikzpicture}
			\end{figure}
			\end{example}

	\medskip		
	
We finish this section with  simple remarks. Firstly, the   affine length of a fully commutative element $w$ written in form  (\ref{formefinale}) is $p + k$. Secondly, the 
  family of elements corresponding to the case $k>0$ and $j=n$ is particularly simple: their form is 
$$h (i,n+1) \;  a_{n+1}    (h(n,n+1) \; a_{n+1} )^t  
\; w_r   $$ for some $i$  such that $0\le i \le n$ and with $w_r= 1$ or 
$  \lfloor n, l \rfloor $ with $n \ge l \ge 1$.   \\

\section{Fully commutative affine braids} 

We denote by $B(\tilde A_{n} )$   the affine braid  group with $n+1$ generators of type $\tilde{A}$, while we denote by $B(A_{n})$   the braid   group with $n$ generators of type $A$, where $n \geq 1 $. 
		By definition $B(\tilde A_{n} ) $ has $ \left\{ \sigma_{1}, \sigma_{2}, \dots,   \sigma_{n}, a_{n+1}   \right\}$ as a set of generators together with the following defining relations:  \\
		        	
			\begin{itemize}[label=$\bullet$, font=\normalsize, font=\color{black}, leftmargin=2cm, parsep=0cm, itemsep=0.25cm, topsep=0cm]
				 \item[(1)] $\sigma_{i} \sigma_{j} =\sigma_{j} \sigma_{i} $  for $1\leq i,j\leq n$ and $ \left| i-j\right| \geq 2$,
				 \item[(2)] $\sigma_{i}\sigma_{i+1}\sigma_{i} = \sigma_{i+1}\sigma_{i}\sigma_{i+1}$ for $1\leq i\leq n-1$,
				 \item[(3)] $\sigma_{i} a_{n+1} = a_{n+1} \sigma_{i} $ for $2\leq i \leq n-1$, 
				 \item[(4)] $\sigma_{1} a_{n+1} \sigma_{1} = a_{n+1} \sigma_{1} a_{n+1}$ for $n\geq 2$,
				 \item[(5)] $\sigma_{n} a_{n+1} \sigma_{n} = a_{n+1} \sigma_{n} a_{n+1}$ for $n\geq 2$,\\
			\end{itemize}
while $B(A_{n})$ is   generated by $ \left\{ \sigma_{1}, \sigma_{2}, \dots,  \sigma_{n}  \right\}$ with relations (1) and (2). 

\begin{lemma}\label{braidinjection} The following map: 
       \begin{eqnarray}
					R_{n}: B(\tilde A_{n-1} ) &\longrightarrow& B(\tilde A_{n} )  \nonumber\\
					\sigma_{i} &\longmapsto& \sigma_{i}$ ~~~ \text{for} $1\leq i\leq n-1 \nonumber\\
					a_{n} &\longmapsto& \sigma_{n} a_{n+1}\sigma^{-1}_{n} \nonumber
				\end{eqnarray}
				is a group monomorphism. 
\end{lemma} 
\begin{proof} 
 Graham and Lehrer give in   \cite[\S 2]{Graham_Lehrer_2003}  a presentation of the $B$-type braid group $B(B_{n+1})$ with generators 
$ \left\{\tau_{n+1},  \sigma_{1},   \dots,   \sigma_{n}, a_{n+1}   \right\}$ and relations 
(1) to (5) above, plus 
$$ \tau_{n+1} \sigma_i \tau_{n+1}^{-1} = \sigma_{i+1} \text{ for } 1\le i \le n-1, \quad 
\tau_{n+1} \sigma_n \tau_{n+1}^{-1} =a_{n+1}  , \quad  \tau_{n+1} a_{n+1} \tau_{n+1}^{-1} = \sigma_{1}. $$
This presentation is related to the usual presentation of the  $B$-type braid group,   with generators 
$ \left\{t,  \sigma_{1}, \sigma_{2}, \dots,   \sigma_{n}     \right\}$ and braid relations, by 
$$\tau_{n+1} = t   \sigma_{1}   \sigma_{2}  \dots  \sigma_{n}, \quad  
a_{n+1} = \tau_{n+1} \sigma_n \tau_{n+1}^{-1} =  t   \sigma_{1}   \sigma_{2}  \dots  \sigma_{n}  \sigma_{n-1}^{-1}  \dots   \sigma_{1}^{-1} t^{-1} .$$ 
They show that the subgroup of $B(B_{n+1})$ generated by $ \left\{   \sigma_{1},   \dots,   \sigma_{n}, a_{n+1}   \right\}$ is isomorphic to $B(\tilde A_{n} )$ and fits into the exact sequence \cite[Corollary 2.7]{Graham_Lehrer_2003}
 $$1\longrightarrow B(\tilde A_{n} )  \longrightarrow B(B_{n+1}) \longrightarrow \mathbb{Z}\longrightarrow 1.$$ 
We thus view  $B(\tilde A_{n} )$ as a subgroup of the  braid group $B(B_{n+1})$ for $n \geq 1$.   

There is a natural injection from $B(B_{n})$ generated by  $ \left\{t,  \sigma_{1}, \sigma_{2}, \dots,   \sigma_{n-1}     \right\}$ into $ B(B_{n+1})$ generated by   $ \left\{t,  \sigma_{1}, \sigma_{2}, \dots,   \sigma_{n}     \right\}$  (see e.g.  \cite[\S 2.2]{Geck_Lambro}).  We claim 
    that this injection restricts to $R_n$ on $ B(\tilde A_{n-1}  )$. To prove this claim we only need to check that  the image of $a_n$ is $\sigma_{n} a_{n+1}\sigma^{-1}_{n} $. Since  $\sigma_n$ commutes with $t$, $\sigma_1$, \dots,  $\sigma_{n-2}$, we have: 
$$\begin{aligned}
\sigma_{n} a_{n+1}\sigma^{-1}_{n} &= 
\sigma_{n} t   \sigma_{1}   \sigma_{2}  \dots  \sigma_{n}  \sigma_{n-1}^{-1}  \dots   \sigma_{1}^{-1} t^{-1}\sigma^{-1}_{n} \\ 
&=  t   \sigma_{1}   \sigma_{2}  \dots \sigma_{n-2}  \sigma_{n}  \sigma_{n-1} \sigma_{n}  \sigma_{n-1}^{-1} \sigma^{-1}_{n} \sigma_{n-2}^{-1} \dots   \sigma_{1}^{-1} t^{-1} 
\\ 
&=  t   \sigma_{1}   \sigma_{2}  \dots \sigma_{n-2}  \sigma_{n-1}  \sigma_{n } \sigma_{n-1}  \sigma_{n-1}^{-1} \sigma^{-1}_{n} \sigma_{n-2}^{-1} \dots   \sigma_{1}^{-1} t^{-1} 
\\ 
&=  t   \sigma_{1}   \sigma_{2}  \dots \sigma_{n-2}  \sigma_{n-1}    \sigma_{n-2}^{-1} \dots   \sigma_{1}^{-1} t^{-1} 
\\ 
&= a_n
\end{aligned}
$$     as announced. 
\end{proof} 		
		
Using   the injective morphism $R_n$ we now view $ B(\tilde A_{n-1} )$ as a subgroup of   $ B(\tilde A_{n} )$. 
Let  $\psi_{n} $ be   the Dynkin automorphism  of $ B(\tilde A_{n-1} ) $ that shifts the generators of the Dynkin diagram one step   counterclockwise ($ \sigma_{1} \mapsto a_{n} \mapsto   \sigma_{n-1}   \mapsto \dotsb   \mapsto  \sigma_{2}   \mapsto \sigma_{1} $).
It generates a subgroup of $\text{Aut}( B(\tilde A_{n-1} ) )$ of order $n$.   We  simply refer  to it by  $\psi$ in what follows. 
		
\begin{lemma}\label{Dynkin}    The element 
   $c_n=  \sigma_{n} \sigma_{n-1}\dots \sigma_{1}a_{n+1} $ of $ B(\tilde A_{n} ) $   normalizes $ B(\tilde A_{n-1} ) $. It acts  by conjugacy on $ B(\tilde A_{n-1} ) $ in the same way as $\psi$.   We will write: 
   $$
   c_n^t h = \psi^{t} \left[ h\right]   c_n^t  \text{ for any } h \in  B(\tilde A_{n-1} )  . 
   $$ 
\end{lemma}  
\begin{proof} Applying braid relations  in $ B(\tilde A_{n} ) $ we see that, for $ 2 \leq i \leq n-1$, we have: 
			\begin{eqnarray}
				\sigma_{n} \sigma_{n-1} \dots  \sigma_{1}a_{n+1} \sigma_{i} &=& 
\sigma_{n}  \dots \sigma_{i+1}\ \sigma_{i}\sigma_{i-1}\sigma_{i} \dots \sigma_{1}a_{n+1} \nonumber \\ 
 &=& 
\sigma_{n}  \dots \sigma_{i+1}\sigma_{i-1}\sigma_{i} \sigma_{i-1}\dots \sigma_{1}a_{n+1}\nonumber \\ 
 &=&
\sigma_{i-1} \sigma_{n}\sigma_{n-1} \dots \sigma_{1}a_{n+1}  .  \nonumber
\end{eqnarray}
In a similar way, we use the braid relations involving $a_{n+1} $ and the following relations worth keeping in mind: 
\begin{equation}	\label{braidnnplus1}		  a_{n+1} a_{n}=  \sigma^{-1}_{n} a_{n}\sigma_{n}a_{n} = \sigma^{-1}_{n} \sigma_{n} a_{n} \sigma_{n}= a_{n} \sigma_{n} =  \sigma_{n}a_{n+1}    
\end{equation} 
to obtain: 
\begin{eqnarray}
				\sigma_{n} \sigma_{n-1} \dots \sigma_{1}a_{n+1} \sigma_{1} &=& a_{n} \sigma_{n} \sigma_{n-1} \dots\sigma_{1}a_{n+1}, \nonumber\\
				\sigma_{n} \sigma_{n-1} \dots \sigma_{1}a_{n+1} a_{n} &=& \sigma_{n-1} \sigma_{n} \sigma_{n-1} \dots \sigma_{1}a_{n+1}  \nonumber
			\end{eqnarray}					
hence our result.			\end{proof} 

Let $n \ge 1$. Recall from \cite[IV.1, Proposition 5]{Bourbaki_1981} that there exists a map 
$$g :   W (\tilde A_{n} )  \longrightarrow B(\tilde A_{n} ) $$ 
such that, given any  reduced expression of   $w \in W (\tilde A_{n} )$, the image 
$g(w)$ of $w$ is defined by the same expression in the braid group 
(keeping the same symbols for the generators of the affine braid group  and their images via the natural surjection onto the affine Coxeter group).

\begin{definition} 
  We call  {\rm fully commutative braids} the images  under $g$ of the fully commutative elements in $W (\tilde A_{n} )$, 
    i.e. the elements of 
   $\{g(w) \ | \   w \in  W^{c}(\tilde A_{n} )\}$.
  \end {definition}	

We now proceed to give a general form for fully commutative braids in which the element $c_n$ is singled out, as a consequence of  the normal form for fully commutative elements in   Theorem~\ref{FC}. 	We use the notation  in Section~\ref{notations},   equally valid for the braid groups,  with an additional index $n$ 
in $B (\tilde A_{n} ) $ and $n-1$ in $B  (\tilde A_{n-1} ) $.

		\begin{lemma} \label{3_1} Consider the element 
			  $ \  y = h_n(j, j+1)  \, a_{n+1}$ of $ \,  B(\tilde A_{n} ) $ 
			  where $1 \leq j \leq n$. Let $ k\geq 1 $ be an integer and write $k= m(n-j+1) +r  \  $ with $  \   0\leq r < n-j+1 $. Then if $j=n$ we have 
			  $y^k= c_n^{k}$ while for $j<n$ we have (with the convention 
			  $\lfloor n,n+1-r \rfloor =1$ if  $r=0$):  \\

			\begin{itemize}[label=$\bullet$, font=\normalsize, font=\color{black}, leftmargin=2cm,parsep=0cm, itemsep=0.25cm, topsep=0cm]
				\item[$(1)$] if $ m=0 $:
					$  \   y^{k} = (h_{n-1}(j, j+1)a_{n})^{r} \lfloor n,n+1-r \rfloor ; $	
	 				 \item[$(2)$] if $ m>0 $:
$$ \begin{aligned}
y^{k}  &=    \left(\prod^{i=m-1}_{i=0}  \psi^{i} \left[ (h_{n-1}(j, j+1)a_{n})^{n-j}\lceil  j,n-1 \rceil 	 \right] \right) \psi^{m}\big[(h_{n-1}(j, j+1)a_{n})^{r}\big]  \\						  
						&\qquad \qquad \qquad  c_n^{m}  \lfloor n,n+1-r \rfloor. \end{aligned}
$$	
			\end{itemize}
		\end{lemma}
				
		\begin{demo}	We have 	for $j<n$, using relations (\ref{braidnnplus1}	): 
			$$
				 y = \sigma_{j} \dots \sigma_{1} \sigma_{j +1 } \dots \sigma_{n-1} \sigma_{n}a_{n+1} =  \sigma_{j} \dots  \sigma_{1} \sigma_{j +1 } \dots \sigma_{n-1} a_{n} \sigma_{n}= h_{n-1}(j, j+1)a_{n}\sigma_{n}. $$
			 Observe that for $ j+1  < s \leq n $  we have $ \sigma_{s}y= y \sigma_{s-1}$. Hence if $j+1 < n$: 
$$				 y^{2} =  h_{n-1}(j, j+1)a_{n}\sigma_{n}  y = 
h_{n-1}(j, j+1)a_{n}y \sigma_{n-1}  = (h_{n-1}(j, j+1)a_{n})^2  \sigma_{n}\sigma_{n-1}.	    $$
			Continuing this way, we can see that  whenever $ 0 \leq t \leq n-j $ we have 
	$$
				y^{t}= (h_{n-1}(j, j+1)a_{n})^{t} \lfloor n,n+1-t \rfloor.  $$
 This holds   in particular for 		$t=n-j$   thus: 
$$
				y^{n-j+1} = (h_{n-1}(j, j+1)a_{n})^{n-j}
				\sigma_{n} \sigma_{n-1} \dots  \sigma_{j+1}\sigma_{j} \dots \sigma_{2} \sigma_{1} \sigma_{j +1 } \dots \sigma_{n-1} \sigma_{n} a_{n+1}   . 
$$				
We now use:   
$ \  \sigma_{n} \sigma_{n-1} \dots   \sigma_{1} \sigma_u = 
\sigma_{u-1} \sigma_{n} \sigma_{n-1} \dots   \sigma_{1} \ $ 
for $2\le u \le n$,  and get: 
$$
			y^{n-j+1} = (h_{n-1}(j, j+1)a_{n})^{n-j} \lceil  j,n-1 \rceil c_n. $$		
An easy induction using Lemma~\ref{Dynkin} leads to:   
$$
				y^{m(n-j+1)} = \left( \prod^{i=m-1}_{i=0} \psi^{i} \left[  (h_{n-1}(j, j+1)a_{n})^{n-j} \lceil  j,n-1 \rceil \right] \right)   \  c_n^{m} .$$
			
			Finally, let $k= m(n-j+1) +r $, where $0\leq r < n-j+1 $. We have:
		$$
		\begin{aligned}
				&y^{k}  =    \left(\prod^{i=m-1}_{i=0} \psi^{i} \left[  (h_{n-1}(j, j\!+\!1)a_{n})^{n-j} \lceil  j,n-1 \rceil \right] \right)   \  c_n^{m}
				(h_{n-1}(j, j\!+\!1)a_{n})^{r} \lfloor n,n+1-r \rfloor 
			 \\
			   &= \!  \left(\prod^{i=m-1}_{i=0} \psi^{i} \left[  (h_{n-1}(j, j\!+\!1)a_{n})^{n-j} \lceil  j,n-1 \rceil \right]  \right)   \  \psi^{m}\left[(h_{n-1}(j, j\!+\!1)a_{n})^{r} \right] c_n^{m}
				 \lfloor n,n\!+\!1\!-\! r \rfloor  				 \end{aligned} 	
$$	as claimed. 		
		\end{demo}

			\begin{theorem} \label{2_5_2}
			Let $ n \ge 2$. Let $w$ be a fully commutative braid in $B(\tilde A_{n} )$. Then $w$ can be written in one and only one of the following two forms:

			\begin{eqnarray}
				&   &u \ c_n^{t} \ v   
				   \nonumber \\  
				\text{or } \quad   &  \lfloor i,1 \rfloor  a_{n+1} \  &u \ c_n^{t} \  v , 
			     \nonumber  \\ \nonumber
			\end{eqnarray} 			
with  $u \in B(\tilde A_{n-1} )$,   $v \in B(A_{n})$,  $t \ge 0$   and $ 0 \leq i  \leq n-1$.

		\end{theorem}

\begin{proof} We have $w= g(\bar w)$  with  $\bar{w}$   in $ W^{c}(\tilde A_{n} )$,  
		and we use the reduced expression of $w$ corresponding to    the normal form of  $\bar w$   given by  Theorem~\ref{FC},  of which we use the notations. 
			We may and do ignore the rightmost term $w_r$ that belongs to $B(A_{n})$ and we remark that if $p=0$ the element has indeed the first claimed form,  
			  by Lemma~\ref{3_1}. We proceed with $p>0$ and 
			set 
				$$ x  = h_n(i_1, r_1) \; a_{n+1} \dots h_n(i_p, r_p) \;  a_{n+1}  $$ 
so that  $w =  x  y^k   $  where   			$ \  y = h_n(j, j+1)  \, a_{n+1}$ as in the previous lemma. \\

We remark that for $1\le i \le p$ we have $ r_{i} \leq r_{1} - i +1$; since    $ r_{1} \leq n+1 $  this  gives: $  r_{i} \leq n- i+2 $. 	
		We have  two   main cases to consider: 		\\ 	
		\begin{itemize}[label=$\bullet$, font=\normalsize, font=\color{black}, leftmargin=2cm,parsep=0cm, itemsep=0.25cm, topsep=0cm]
			\item   $ r_{1} \leq n $, that is $\sigma_{n}$ belongs to the support of $\sigma_{r_{1}} \dots \sigma_{n-1} \sigma_{n} $ ; 
			\item   $ r_{1} = n+1 $, that is  $\sigma_{n}$ does not belong to the support of $  \sigma_{r_{1}} \dots \sigma_{n-1} \sigma_{n} $. \\  		\end{itemize} 
				
		We start with the first case  $ r_{1} \leq n $, so that $ r_{i} \leq n-i+1 $. 
		Using first the relation $ \sigma_{n}a_{n+1} = a_{n} \sigma_{n}  $, then 
 the braid relations $\sigma_u \sigma_{u-1} \sigma_u = \sigma_{u-1} \sigma_u 
 \sigma_{u-1} $ for $u=n$, $n-1$, \dots, $n-p+2$,  
   we push to the right the first $\sigma_n$ from the left   and get:  
$$
x= \lfloor i_1,1 \rfloor  \lceil  r_1,n-1 \rceil a_n  h(i_2, r_2) \; a_{n+1} \dots h(i_p, r_p) \;  a_{n+1} \sigma_{n-(p-1)} . 
$$ 
We  repeat this operation with the first $\sigma_n$ from the left in this new expression, and so on. 
Proceeding from $i=1$ to $i=p$,   we obtain:
$$
x= h_{n-1}(i_1, r_1) \; a_{n } \dots h_{n-1}(i_p, r_p) \; a_{n } 
\sigma_{n}\sigma_{n-1} \dots \sigma_{n - (p-1)}.
$$ 		
	
We set  $\rho                                                                                                                                                                                                                                                                                                                                                                                                                                                                                                                                                                                                                                                                                                                                                                                                                                                                                                                                                                                                                                                                                                                                                                                                                                                     = 	h_{n-1}(i_1, r_1) \; a_{n } \dots h_{n-1}(i_p, r_p) \; a_{n } $, an element of $B(\tilde A_{n-1} ) $. 
	If $ k =0 $ we have  $w=x= \rho  \lfloor n ,n - (p-1)\rfloor $ and our claim holds. Let $ k \ge 1 $ 
and set $ \epsilon = n - (p-1) $. We have $ \epsilon \geq j+1 $ since 
		$
			j+1 \leq r_{p} \leq n-(p-1)$ and  we can write:
$$ w = \rho \  \sigma_{n}\sigma_{n-1} \dots  \sigma_{\epsilon}
			\  y^{k}  , \text{ with } \rho \in B(\tilde A_{n-1} ) \text{ and } \epsilon \ge  j+1. 
			$$
 We recall that   $y= h_n(j, j+1)  \, a_{n+1}$    acts on $ \sigma_{i} $ in the following way: $ \sigma_{i} y = y \sigma_{i-1} $  for  $ j+1 < i \leq n$.  We distinguish  two   cases:\\ 
		
		\begin{itemize}[label=$\bullet$, font=\normalsize, font=\color{black}, leftmargin=2cm,parsep=0cm, itemsep=0.25cm, topsep=0cm]
			\item[$(1)$] $1 \leq k \leq \epsilon - (j+1)$, 
			\item[$(2)$] $\epsilon - (j+1) < k $.\\ 
		\end{itemize}
			
		We start with (1).   We have:
		$  \  	\sigma_{n} \sigma_{n-1} \dots \sigma_{\epsilon} y^{k} = y^{k} \sigma_{n-k} \sigma_{n-1-k} \dots \sigma_{\epsilon - k }  $,  
with  $$ k \leq \epsilon - (j+1) =   n-j-p < n-j .$$ Thus we are in case (1) of Lemma~\ref{3_1}, that is:  
$$
w = \rho (h_{n-1}(j, j+1)a_{n})^{k} \lfloor n,n+1-k \rfloor  \lfloor n-k,\epsilon -k \rfloor 
  \   \in \   B(\tilde A_{n-1} ) \ B( {A_{n}}) $$
 as required (this is the first form in the theorem with $t=0$).  \\
	
		Now we deal with case (2) and set  
		 $h = k-(\epsilon - (j+1)) \ge 1$.  Using the computation from case (1) for 
		$\epsilon - (j+1)$  	we get, using again Lemma~\ref{3_1} (1):
	$$\begin{aligned}
			  \sigma_{n}\sigma_{n-1} \dots  \sigma_{\epsilon} \ y^{k} &= \sigma_{n} \sigma_{n-1} \dots \sigma_{\epsilon} \   y^{\epsilon - (j+1)} y^{h}    = y^{\epsilon- (j+1)} \sigma_{j+p}\sigma_{j+p-1} \dots \sigma_{j+1}
		\  	y^{h} \\
			&=  	 (h_{n-1}(j, j+1)a_{n})^{\epsilon - (j+1)} \lfloor n,j+p+1 \rfloor 
			 \lfloor j+p, j+1 \rfloor  	\  	y^{h}  \\   
		&=  	 (h_{n-1}(j, j+1)a_{n})^{\epsilon - (j+1)} \lfloor n, j+1 \rfloor  
			\  y \ 	y^{h-1}	. 
			 \end{aligned} 	$$ 
We now compute: 
$$\begin{aligned}
\lfloor n, j+1 \rfloor  
			\  y \ 	y^{h-1}	
&= ( \sigma_{n} \dots \sigma_{j+1}) ( \sigma_{j} \dots \sigma_{1} ) (\sigma_{j +1}   \dots \sigma_{n}) a_{n+1}y^{h-1}  \\  
&=(\sigma_{n}   \dots \sigma_{1}) (  \sigma_{j+1} \dots   \sigma_{n})  a_{n+1} y^{h-1} \\
&= (\sigma_{j} \dots  \sigma_{n-1}) (\sigma_{n} \dots  \sigma_{1}) a_{n+1} y^{h-1} 
 \end{aligned} $$ 	 
since 	 $( \sigma_{n} \dots \sigma_{ 1}) \sigma_k= \sigma_{k-1} ( \sigma_{n} \dots \sigma_{ 1})$ 
for $2 \le k \le n$.   
Setting $$\eta = \rho  (h_{n-1}(j, j+1)a_{n})^{\epsilon - (j+1)} 
		 \sigma_{j} \dots   \sigma_{n-1}   \in B(\tilde A_{n-1} )  ,$$   
 	we get  $w = \eta c_n y^{h-1}   $. \\

		\begin{itemize}[label=$\bullet$, font=\normalsize, font=\color{black}, leftmargin=2cm,parsep=0cm, itemsep=0.25cm, topsep=0cm]
			\item[$(a)$] If $ h-1 \leq n-j $, we see, using Lemma~\ref{3_1},  that:
				\begin{eqnarray}
					w &=& \eta c_n  (h_{n-1}(j, j+1)a_{n})^{h-1}
					 \lfloor n, n+1-(h-1) \rfloor  
					  \nonumber \\
					 &=& \eta \  \psi \left[  (h_{n-1}(j, j+1)a_{n})^{h-1} \right] c_n   \lfloor n, n+1-(h-1) \rfloor  
, \nonumber 
				\end{eqnarray}
which is of the first claimed form.     
				
			\item[$(b)$] If $n-j< h-1 $, we write $h-1= m(n-j+1) +r$ 
			with $0\leq r < n-j+1 $ as in Lemma~\ref{3_1} and get:			 
			$$\begin{aligned}
					w &=  \eta c_n \left( \prod^{i=m-1}_{i=0}  \psi^{i} \left[ (h_{n-1}(j, j+1)a_{n})^{n-j}\lceil  j,n-1 \rceil 	 \right] \right) \psi^{m}\big[(h_{n-1}(j, j+1)a_{n})^{r}\big] \\						 
					&\qquad \qquad \qquad 	c_n^{m}  \lfloor n,n+1-r \rfloor   \\
					 &=   \eta  \left(  \prod^{i=m}_{i=1}  \psi^{i} \left[ (h_{n-1}(j, j+1)a_{n})^{n-j}\lceil  j,n-1 \rceil 	 \right] \right) \psi^{m+1}\big[(h_{n-1}(j, j+1)a_{n})^{r}\big]  \\						 
				&\qquad \qquad \qquad		c_n^{m+1}  \lfloor n,n+1-r \rfloor  
\end{aligned} 			$$
which has indeed  the required  form (first form in the theorem).  	\\  		
		\end{itemize}

		We are left with  with the second main case  $ r_{1} = n+1 $:   the element   under study  has a normal  form $ w= \lfloor i_1,1 \rfloor a_{n+1} w' $ with $0\leq i_{1} \le  n$. In fact  $ i_{1} = n $ is the case of positive powers of $c_n =\sigma_{n} \dots \sigma_{1}a_{n+1}$ that have the first form in the theorem. For 
		$0\leq i_{1} <  n$ the element $w'$ actually belongs to  one of the previous cases thus we get the second form in the theorem.  	\end{proof} 	

Using the natural surjection of 		 $ B(\tilde A_{n} ) $  onto $ W(\tilde A_{n} ) $,  we deduce from 
Theorem~\ref{2_5_2} two possible forms for fully commutative elements in $ W(\tilde A_{n} ) $. In the above  notation, the image of $v \in B(A_n)$ belongs to $W(A_n)$ and, 
considering the left classes of $W(A_{n-1})$ in $W(A_{n})$ as in   \cite[p.1288]{St}, this image  can be written uniquely as a product 
$v'   \sigma_{n}\sigma_{n-1} \dots  \sigma_{s} $ with $v' \in W(A_{n-1})$ and  $1 \le s \le n+1$ (with the convention that for $s=n+1$, 
the product $\sigma_{n}\sigma_{n-1} \dots  \sigma_{s}$ is   $1$). Since $c_n$ normalizes  $W(\tilde A_{n-1})$ 
the element $v'$ can be moved to the left of $c_n^t$ into an element that incorporates with the image of $u$. 
We obtain:

		\begin{corollary} Let $ w $ be a fully commutative element in $ W(\tilde A_{n} ) $ for  $ n \ge 2 $.
		Then  $ w$ can be written in one  and only one of the following two forms: 
		
			\begin{eqnarray}
				& & d \,  c_n^{t} \sigma_{n}\sigma_{n-1} \dots  \sigma_{s}   \nonumber \\
				\text{or }  \    \lfloor i,1 \rfloor a_{n+1} \, 
				& &  d  \,   c_n^{t} \sigma_{n}\sigma_{n-1} \dots  \sigma_{s},  \nonumber \\  \nonumber
			\end{eqnarray}
	where $d $ is in $W(\tilde A_{n-1} )$ and $t$, $s$  and $i$ are integers with $t \ge 0$, $1 \le s \le n+1$  and $ 0 \leq i  \leq n-1$. \\  
		\end{corollary}

\section{The tower of fully commutative elements}

The injection of braid groups $R_{n}: B(\tilde A_{n-1} )  \longrightarrow  B(\tilde A_{n} )$ of 
Lemma~\ref{braidinjection} induces   the group homomorphism: 
								\begin{eqnarray}
				R_{n}: W(\tilde A_{n-1} ) &\longrightarrow& W(\tilde A_{n} )\nonumber\\
				\sigma_{i} &\longmapsto& \sigma_{i} \text{ for } 1\leq i\leq n-1\nonumber\\
				a_{n} &\longmapsto& \sigma_{n} a_{n+1}\sigma_{n}.  \nonumber
			\end{eqnarray}
			\begin{lemma}\label{Winjection}
The morphism 		$R_{n}: W(\tilde A_{n-1} )  \longrightarrow  W(\tilde A_{n} ) $  is  an injection.  		
			\end{lemma} 
\begin{proof}
This fact is most easily seen using the description of $W(\tilde A_{n})$ as the group of 
$(n+1)$-periodic permutations of $ \mathbb{Z}$ with total shift equal to 0, which we briefly recall.  
			A permutation 	$u$ of $ \mathbb{Z}$  is  $m$-{\rm periodic}   if 
			$$u(i+m) = u(i) +m  \text{ for any } i \in  \mathbb{Z} .$$
			 We define the total shift of an $m$-periodic permutation  	$u$    to be:
				\begin{eqnarray}
					\frac{1}{m} \sum\limits^{i=m}_{i=1} \big(u(i) - i\big).\nonumber
				\end{eqnarray}		 
We set $ ^{m}_{0}\mathbb{Z}$ to be the set of $m$-periodic permutations with total shift equal to 0. It forms a subgroup of the group of permutations of $\mathbb{Z} $. Let $i$ be an integer such that $ 1 \leq i\leq m  $ and let $ ^ms_i$ 
be  the $m$-periodic permutation defined by:  $^ms_i (k)= k+1$ for 
$k \equiv i $ (mod $m$), $^ms_i (k)= k-1$ for $k \equiv i+1 $ (mod $m$), 
$^ms_i (k)= k $ for $k \not\equiv i, i+1$ (mod $m$). 
 	Then 
	\begin{eqnarray}\label{Lusztig} 
			 \sigma_{i} \  &\longmapsto& \    ^{n+1}s_i  \quad  \text{ for } 1\leq i\leq n   
			 \nonumber\\
				a_{n+1} \   &\longmapsto& \   ^{n+1}s_{n+1}  
			\end{eqnarray}
is an isomorphism from $W(\tilde A_{n})$ onto $ ^{n+1}_{0}\mathbb{Z}$ (see e.g. \cite[Proposition 3.2]{Digne}).  \\ 

It is easy to check
 that $ ^{n}_{0}\mathbb{Z}$ injects into $ ^{n+1}_{0}\mathbb{Z}$ as follows. We define an injection 
 $\phi: \mathbb Z \longrightarrow \mathbb Z$ by letting $\phi(i+ kn)= 
 i+k(n+1)$ for $1\leq i\leq n $ and $k \in \mathbb Z$. We then map an $n$-periodic permutation $v$ to the $(n+1)$-periodic permutation $v'$ defined 
 by $v'(\phi(x))= \phi(v(x))$ ($x \in \mathbb Z$) and 
 $v'(k(n+1))= k(n+1)$ ($k \in \mathbb Z$). For $v \in  ^{n}_{0}\mathbb{Z}$ and 
  $1\leq i\leq n $, write $v(i)= j_i + k_i n$ with $1\leq j_i\leq n $;   the total shift of $v$ is  $(\sum\limits^{i=n}_{i=1} k_i )+ \frac{1}{n } \sum\limits^{i=n}_{i=1} \left(j_i - i\right) $.  
Since $v$ is an $n$-periodic permutation   it induces a permutation mod $n$,  hence   $ \sum\limits^{i=n}_{i=1}  j_i    = \sum\limits^{i=n}_{i=1}    i $ and consequently 
    $ \sum\limits^{i=n}_{i=1} k_i =0$. 
  Now  
 the total shift of $v'$ is   
 $$
 \begin{aligned}
 \frac{1}{n+1} \sum\limits^{i=n+1}_{i=1} \left(v'(i) - i\right) &= 			
 	\frac{1}{n+1}\left(  \sum^{i=n}_{i=1} ( \phi(v(i)) - i)\right) + \frac{1}{n+1} (n+1 - (n+1)) 
	\\
	&=  \frac{1}{n+1}  \sum^{i=n}_{i=1} ( j_i + k_i (n+1) - i)   = 0
	\end{aligned}
	$$  
	    This injection maps $^{n }s_i $ to $^{n+1}s_i $  for $1 \leq i \leq n-1$, and   $^{n }s_n$   to $\  ^{n+1 }s_n  \ ^{n+1 }s_{n+1 }   \    ^{n+1 }s_n$. Going back to 
    $W(\tilde A_{n-1})$ and $W(\tilde A_{n})$  through the isomorphisms (\ref{Lusztig}) above, we find the morphism $R_n$, hence itself  an injection.
\end{proof}

The Coxeter group $W(A_{n-1})$ with    Coxeter generators $(\sigma_1, \cdots, \sigma_{n-1})$ is  a parabolic subgroup of $W(A_{n})$. This is no longer the case for $W(\tilde A_{n-1})$ and $W(\tilde A_{n})$,  indeed proper parabolic subgroups of   $W(\tilde A_{n})$ are finite. This 
  is an important difficulty when dealing with the affine case.  $W(  A_{n})$
though, with    Coxeter generators $(\sigma_1, \cdots, \sigma_{n})$,   is a parabolic subgroup of  $W(\tilde A_{n})$ and fully commutative elements 
of $W(  A_{n})$ are fully commutative in $W(\tilde A_{n})$ as well. 
As for $W(\tilde A_{n-1})$, the   injection $R_n: 					  W(\tilde A_{n-1}) \longrightarrow  W(\tilde A_{n}) $ of Lemma~\ref{Winjection} is a group homomorphism,  but it does not preserve full commutativity, 
as can be seen directly on the image of $a_n$, namely $\sigma_{n} a_{n+1}\sigma_{n}$.   When dealing with fully commutative elements, the notion of homomorphism of groups may thus become irrelevant. We introduce the following maps:  \\

\begin{theorem} \label{IJ} For  $w \in W^c(\tilde A_{n-1} )$, let  $I(w)$  (resp.  $J(w)$) be the 
element of $W (\tilde A_{n} ) $ obtained by substituting 
$\sigma_n a_{n+1}$ (resp. $a_{n+1} \sigma_n $)  to $a_n$ in the normal form for $w$.  
Both expressions are reduced and both $I(w)$ and $J(w)$ are fully commutative, with the same affine length as $w$. 
The maps: 

$$I, J : W^c(\tilde A_{n-1} ) \longrightarrow W^c(\tilde A_{n} ) $$

\medskip\noindent
are injective and satisfy 
  $l(I(w)) =  l(J(w)) = l(w) + L(w)$. Furthermore   
  the images of $I$ and $J$ intersect exactly on $  W^c( A_{n-1} )$.  \\ 
\end{theorem}

\begin{proof} We use the notation in Section~\ref{notations} with an additional index $n$ 
in $W (\tilde A_{n} ) $ and $n-1$ in $W (\tilde A_{n-1} ) $. 
We then see that  
$ 
I(h_{n-1}(i,j) \, a_n)= h_n(i,j) \, a_{n+1} $.  \\ 

Let $w \in W^c(\tilde A_{n-1} )$. 
By inspection of the normal forms in $  W^c(\tilde A_{n-1} ) $ and $ W^c(\tilde A_{n} ) $ 
given in Theorem~\ref{FC} one sees directly that $I(w)$  is the normal form of a fully commutative element in  $ W(\tilde A_{n} ) $ and that this process is injective. Indeed we have: 
$$
\begin{aligned}   
I( h_{n-1}(i_1, r_1) \; a_{n} \dots &h_{n-1}(i_p, r_p) \;   a_{n}    (h_{n-1}(j,j+1) \; a_{n} )^k  
\;    w_r  )  \\
&=    h_{n}(i_1, r_1) \; a_{n+1} \dots h_{n}(i_p, r_p) \;  a_{n+1}    (h_{n}(j,j+1) \; a_{n+1} )^k  
\;    w_r. 
\end{aligned} $$

\medskip
As for $J(w)$,   the corresponding normal form is obtained as follows. We first observe that 
$$
J(h_{n-1}(j,j+1) \; a_{n })=   \lfloor j,1 \rfloor  \lceil  j+1,n-1 \rceil  a_{n+1} \sigma_n 
=   \lfloor j,1 \rfloor a_{n+1}   \lceil  j+1,n  \rceil  ,  
$$
which implies,  since $\sigma_j, \dots,\sigma_1$ commute  with $\sigma_{j+2}, \dots, \sigma_n$: 
$$
\begin{aligned} 
J((h_{n-1}(j,j+1) \; a_{n })^2)      
&=   \lfloor j,1 \rfloor a_{n+1}   \lceil  j+1,n  \rceil    \lfloor j,1 \rfloor a_{n+1}   \lceil  j+1,n  \rceil \\ &=   \lfloor j,1 \rfloor a_{n+1}   \lfloor j+1,1 \rfloor \lceil  j+2,n  \rceil    a_{n+1}   \lceil  j+1,n  \rceil ,
\end{aligned}
$$
 and 
inductively:
$$
\begin{aligned} 
J((h_{n-1}(j,j+1) \; a_{n })^k)      
&=      \lfloor j,1 \rfloor a_{n+1}  \left( h_n( j+1,  j+2)    a_{n+1} \right)^{k-1}  \lceil  j+1,n  \rceil . 
\end{aligned}
$$

\medskip 
For $k>0$ we thus get 
\begin{equation}\label{Jw1}
  J((h_{n-1}(j,j+1) a_{n } )^k  
     w_r ) =   h_n(j,n\!+\!1)   a_{n+1}  \left( h_n( j\!+\!1,   j\!+\!2)    a_{n+1} \right)^{k-1}  \lceil  j\!+\!1,n  \rceil w_r 
\end{equation}
\noindent 
and we observe that either $w_r=1$, in which case  $\lceil  j+1,n  \rceil$ has the shape required by Theorem \ref{FC}, or  $w_r=  \lfloor j, d_1 \rfloor \lfloor j+1, d_2 \rfloor \dots  \lfloor j+z-1, d_{z} \rfloor $, in which case  
$$ \lceil  j+1,n  \rceil w_r =  \lfloor j+1, d_1 \rfloor \lfloor j+2, d_2 \rfloor \dots  \lfloor j+z, d_{z} \rfloor \lceil  j+z+1,n  \rceil $$
which is again the shape of the rightmost element required in Theorem~\ref{FC}. 
The expression (\ref{Jw1}) is thus the normal form of $J(w)$ when $p=0$ in Theorem~\ref{FC}
(\ref{formefinale}). \\

If $p$ is positive in  Theorem~\ref{FC}
(\ref{formefinale}), we have:

$$\begin{aligned}
J( h_{n-1}(i_p, r_p) \;   a_{n}  ) \lfloor j,1 \rfloor   a_{n+1} &=  
 \lfloor i_p,1 \rfloor  \lceil  r_p ,n-1 \rceil  a_{n+1} \sigma_n  \lfloor j,1 \rfloor a_{n+1}     \\
 &=   \lfloor i_p,1 \rfloor  a_{n+1} \lceil  r_p ,n \rceil  \lfloor j,1 \rfloor   a_{n+1}
 . 
 \end{aligned}
$$
Thus  if $j < r_p-1$: 
$$\begin{aligned}
J( h_{n-1}(i_p, r_p) \;   a_{n}  ) \lfloor j,1 \rfloor   a_{n+1}&=  
   \lfloor i_p,1 \rfloor  a_{n+1} \lfloor j,1 \rfloor   \lceil  r_p ,n \rceil   a_{n+1}
  \\
 &= \lfloor i_p,1 \rfloor  a_{n+1} h_n( j ,   r_p)    a_{n+1}  
 \end{aligned}
$$
while if $j = r_p-1$: 
$$\begin{aligned}
J( h_{n-1}(i_p, r_p) \;   a_{n}  ) \lfloor j,1 \rfloor   a_{n+1}&=  
   \lfloor i_p,1 \rfloor  a_{n+1} \lfloor r_p,1 \rfloor   \lceil  r_p+1 ,n \rceil   a_{n+1}
   \\ 
 &=   \lfloor i_p,1 \rfloor  a_{n+1}   h_n( j+1,   j+2)    a_{n+1}  
  . 
 \end{aligned}
$$

We proceed in the same way from right to left, noticing that when going from $(i_t, r_t)$ 
to $(i_{t-1}, r_{t-1})$ the case $i_t = r_{t-1}-1$ cannot occur, so we get: 
$$\begin{aligned}
J( h_{n-1}(i_{t-1}, r_{t-1}) \;   a_{n}  ) \lfloor i_t,1 \rfloor   a_{n+1}&=  
   \lfloor i_{t-1},1 \rfloor  a_{n+1} h_n( i_t ,   r_{t-1})    a_{n+1} .  
 \end{aligned}
$$

\medskip 
We can now write the normal form of $J(w)$   when $p>0$ in Theorem~\ref{FC}
(\ref{formefinale}): \\

\noindent 
$J ( h_{n-1}(i_1, r_1) \;  a_{n} \dots  h_{n-1}(i_p, r_p) \;   a_{n}    (h_{n-1}(j,j+ 1) \; a_{n} )^k  
\;    w_r  ) \   $ is equal to: 

$$
\begin{aligned} 
   h_{n}(i_1, n\!+\!1)   a_{n+1} \dots  h_{n}(i_p, r_{p-1})   a_{n+1}  h_{n}(j, r_p)   a_{n+1}    (h_{n}&(j\!+\!1,j\!+\!2)   a_{n+1} )^{k-1}  
    \lceil  j\!+\!1,n  \rceil
  w_r  
  \\    &\text{ if }  k > 0 \text{ and } j < r_p-1,
\\
  h_{n}(i_1, n\!+\!1)   a_{n+1} \dots  h_{n}(i_p, r_{p-1})   a_{n+1}       (h_{n}(j\!+\!1,j\!+\!2)  a&_{n+1}  )^{k}  
    \lceil  j\!+\!1,n  \rceil
  w_r  
\\ 
   &\text{ if }  k > 0 \text{ and }  j = r_p-1, 
\\
 h_{n}(i_1, n\!+\!1)   a_{n+1} \dots  h_{n}(i_p, r_{p-1})   a_{n+1}       
    \lceil  r_p,n  \rceil
  w_r  
  \qquad  \quad  &\text{ if }  k = 0.  
\end{aligned}
$$

\noindent 
We see as before, by a suitable right shift of $\sigma_{r_p}, \dots, \sigma_n$,  that for $k=0$ the rightmost term $ \lceil  r_p,n  \rceil
  w_r $ has again the   shape required by Theorem~\ref{FC}. \\ 
 
The fact that   the substitution process adds to the original length the number of occurrences of $a_n$, i.e. the affine length, is clear. As for the intersection of the images, one only needs to notice that if  a reduced expression of a fully commutative element 
contains an element $\sigma_n$ to the left of the first $a_{n+1}$ from left to right, then 
all reduced expressions for this element have the same property.    
\end{proof} 

\medskip
We remark that the injection $I$ is well defined only on the set of fully commutative elements. Indeed, substituting $\sigma_n a_{n+1}$ to $a_n$ in 
the two reduced expressions $\sigma_{n-1} a_n \sigma_{n-1}$ and 
$a_n \sigma_{n-1} a_n$ gives rise to  different elements of $W (\tilde A_{n} ) $. 
It might be the case that $I $ is well defined on the set of elements for which 
the number of occurrences of $a_n$ in any reduced expression is the same, but as we do not need this, we will not  examine it further.

\section{The tower of Temperley-Lieb algebras}\label{TL}

Let   $K$ be an integral domain of characteristic $0$ and  $q$ be an invertible element in $K$. 
 We mean by algebra in what follows $K$-algebra.
 For $x,y$ in a given ring with identity we define $$V(x,y) = xyx+xy+yx+x+y+1.$$

			For $ n\geq 2$, we define $\widehat{TL}_{n+1} (q)$ to be the algebra with unit given by the set of generators  $\left\{g_{\sigma_{1}}, \dots,  g_{\sigma_{n}}, g_{a_{n+1}}\right\}$, with the following relations (see \cite[0.1, 0.5]{Graham_Lehrer_1998}):\\
	\begin{equation}\label{definingrelations} 	
	  \quad \left\{ \quad  \begin{aligned}
		 &g_{\sigma_{i}} g_{\sigma_{j}} =g_{\sigma_{j}} g_{\sigma_{i}}  \text{ for } 1\leq i,j\leq n \text{ and } \left| i-j\right| \geq 2, \\
		 &g_{\sigma_{i}} g_{a_{n+1}} =g_{a_{n+1}} g_{\sigma_{i}}   \text{ for  }  2\leq i \leq n-1, \\
			  &g_{\sigma_{i}}g_{\sigma_{i+1}}g_{\sigma_{i}} = g_{\sigma_{i+1}}g_{\sigma_{i}}g_{\sigma_{i+1}}  \text{ for }  1\leq i\leq n-1,\\
			   &g_{\sigma_{i}}g_{a_{n+1}}g_{\sigma_{i}} = g_{a_{n+1}}g_{\sigma_{i}}g_{a_{n+1}}  \text{ for } i= 1, n , \\
			  &g^{2}_{\sigma_{i}} = (q-1)g_{\sigma_{i}} +q  \text{ for }  1\leq i\leq n,\\
			& g^{2}_{a_{n+1}} = (q-1)g_{a_{n+1}} +q , \\ 
			  &V(g_{\sigma_{i}},g_{\sigma_{i+1}})=V(g_{\sigma_{1}},g_{a_{n+1}}) = V(g_{\sigma_{n}},g_{a_{n+1}})= 0  \text{  for }1\leq i\leq n-1.  
			  \end{aligned} \right.    
		\end{equation}
		
		\smallskip
We set $\widehat{TL}_{1} (q) = K$. For $n=1$, the 		
			  algebra $\widehat{TL}_{2}(q) $ is generated by two elements: $ g_{\sigma_{1}}, g_{a_{2}}$, with only  Hecke quadratic relations. That is: 			\begin{eqnarray}
	 			g_{\sigma_{1}}^{2} &=& (q-1) g_{\sigma_{1}}  +q \quad  \text{ and }  
				\quad g_{a_{2}}^{2} = (q-1) g_{a_{2}} +q.  \nonumber\\\nonumber
			\end{eqnarray}

 Let $ w $ be a fully commutative element in $ W(\tilde A_{n})$. 
 Pick a   reduced expression of $w$, say 
 $w= s_1 \cdots s_k$ with $s_i \in \left\{ {\sigma_{1}}, ...,{\sigma_{n}}, {a_{n+1}}\right\}$ for $1\le i \le k$. Then $g_w: = g_{s_1} \cdots 
 g_{s_k}$ is a well defined element in $\widehat{TL}_{n+1} (q)$ that  does not depend on the choice of a reduced expression  of $w$, and the set $\left\{ g_{w} \ | \  w \in W^{c}(\tilde A_{n} )\right\}$ is   a $K$-basis  
		 of $\widehat{TL}_{n+1} (q)$ \cite[Proposition 1]{Fan_1996}.   The multiplication associated to this basis   satisfies, for $w,v$ in $W^{c}(\tilde A_{n} )$ and $s$ in $\left\{\sigma_{1},..., \sigma_{n}, a_{n+1}\right\}$:		 
		\begin{eqnarray}
			g_{w} g_{v} &=& g_{wv} ~~~~~~~~~~~~~~~~~~~~~~~~\text{whenever }  l(wv) = l(w) + l (v) \text{ and } w v \in W^{c}(\tilde A_{n} ) ,\nonumber\\
			g_{s} g_{w} &=&  (q-1) g_{w} +q g_{sw} ~~~~~~~  \text{whenever } l(sw) = l(w) - 1.  \nonumber
		\end{eqnarray}

		The classical Temperley-Lieb algebra of type $A$ with $n$ generators, $TL_{n}(q)$, can be regarded as the subalgebra of $\widehat{TL}_{n+1} (q)$ generated by  $\left\{g_{\sigma_{1}} ,...,~ g_{\sigma_{n}}\right\}$. A $K$-basis 
		of  $TL_{n}(q)$ is given by  $\left\{ g_{w} \ |  \   w \in W^{c}(A_{n})\right\}$. We set $TL_{0}(q) = K$. \\

	The affine Temperley-Lieb algebra $\widehat{TL}_{n+1} (q)$ is the quotient of the affine 	braid group algebra $K[B(\tilde A_{n} )]  $ by the ideal generated by Relations~(\ref{definingrelations}).  		The injection of braid groups $R_{n}: B(\tilde A_{n-1} )  \longrightarrow  B(\tilde A_{n} )$ of Lemma~\ref{braidinjection}  induces the  injection of group algebras:   	 
	\begin{eqnarray}
					G_{n}: K[B(\tilde A_{n-1} )] &\longrightarrow& K[B(\tilde A_{n} )]  \nonumber\\
					\sigma_{i} &\longmapsto& \sigma_{i}$ ~~~ \text{for} $1\leq i\leq n-1 \nonumber\\
					a_{n} &\longmapsto& \sigma_{n} a_{n+1}\sigma^{-1}_{n}  \nonumber
				\end{eqnarray}	
   which, as we will see shortly,  induces an algebra homomorphism at the Temperley-Lieb level. Yet   the possible lack of injectivity of this homomorphism forces us to use different notations for the generators of $\widehat{TL}_{n}(q)$ and $\widehat{TL}_{n+1}(q)$. In the following proposition we use 
	  $\left\{t_{\sigma_{1}}, ..., t_{\sigma_{n-1}}, t_{a_{n}}\right\}$
as the	set of generators for $\widehat{TL}_{n}(q)$ satisfying relations~(\ref{definingrelations})  (with  $t $ replacing   $g $   and $n$ replacing $n+1$).   
 \\

	\begin{proposition}\label{morphismGn}  
	The injection $G_{n}$ induces the following morphism of algebras:	
			\begin{eqnarray}
				R_{n}: \widehat{TL}_{n}(q) &\longrightarrow& \widehat{TL}_{n+1}(q) \nonumber\\
				t_{\sigma_{i}} &\longmapsto & g_{\sigma_{i}} \text{ for } 1 \leq i \leq n-1 \nonumber\\
				t_{a_{n}} &\longmapsto & g_{\sigma_{n}} g_{a_{n+1}} g^{-1}_{\sigma_{n}}. \nonumber
			\end{eqnarray}
The restriction of $R_n$ to $  TL_{n-1}(q)$ is an injective morphism into $TL_{n}(q)$  and satisfies $R_n(t_w)= g_{I(w)}= g_{J(w)}$  for  $w \in W^c( A_{n-1})$. 
	\end{proposition}	 
	
\begin{proof} We first have to show that the defining relations (\ref{definingrelations})   
for $\widehat{TL}_{n}(q)$,  with  $t $ replacing   $g $   and $n$ replacing $n+1$,  are preserved for the images in $TL_{n}(q)$. This is immediate for  those relations that do not involve $a_n$,  using directly 
 relations~(\ref{definingrelations}) for $\widehat{TL}_{n+1}(q) $.  For the others, this is easily checked. 
For instance we have:  
$$ \begin{aligned} 
&R_n(V(t_{\sigma_{1}} , t_{a_{n}})) =  g_{\sigma_{n}} 
V(      g_{\sigma_{1}} , g_{a_{n+1}})        g^{-1}_{\sigma_{n}} =0,  \\
 &R_n(V(t_{\sigma_{n-1}} , t_{a_{n}})) =
  g_{\sigma_{n}} g_{\sigma_{n-1}} 
V(      g_{\sigma_{n}} , g_{a_{n+1}})   g^{-1}_{\sigma_{n-1}}      g^{-1}_{\sigma_{n}} =0 . 
\end{aligned}
$$
The last assertion comes from the fact that $W ( A_{n-1})$ is parabolic in 
$W  ( A_{n })$: fully commutative elements of type $A_{n-1}$ embed in those of type $A_n$ and this embedding preserves the normal form. 
\end{proof} 
 
	\begin{remark}	The injection $G_{n}: K[B(\tilde A_{n-1} )]  \longrightarrow  K[B(\tilde A_{n} )]  $ also induces a  homomorphism of   the corresponding Hecke algebras. We have shown in \cite[Proposition 4.3.3]{Sadek_Thesis} that this homomorphism  is injective for $K=\mathbb Z[q, q^{-1}]$ where $q$ is an indeterminate. We will not  need 
this fact in what follows. \\
	\end{remark}

\begin{proposition}\label{formula} 
For any $w \in  W^c(\tilde A_{n-1})$ of positive affine length,  the element $R_n(t_w)$   has the following form: 
$$
R_n(t_w) =   (-1)^{L(w)}   g_{I(w)} +  (-\frac{1}{q})^{L(w)}   g_{J(w)} 
+  \sum_{\begin{smallmatrix} l(x)< l(I(w)) \cr L(x) \le L(w) \end{smallmatrix}}  \alpha_x g_x  \qquad (\alpha_x \in K). 
$$
\end{proposition}
\begin{proof} We prove the statement by induction on the affine length 
$L(w)$ of an  element  $w\in  W^c(\tilde A_{n-1})$. 
We first assume that  $L(w)=1$. We can write $w = u a_n v$  where 
$u$ and 
 $v$ belong  to $ W^c( A_{n-1})$. 
We compute: 
\begin{eqnarray}
 R_n  (t_{a_n})   	
&=& g_{\sigma_{n}} g_{a_{n+1}} g^{-1}_{\sigma_{n}}  
 =
  g_{\sigma_{n}} g_{a_{n+1}} \left( \frac{1}{q} g_{\sigma_{n}} + (\frac{1}{q}-1)  g_1  \right)   
  \nonumber\\   
&=&  
  \frac{1}{q} g_{\sigma_{n}} g_{a_{n+1}}   g_{\sigma_{n}} + (\frac{1}{q}-1) g_{\sigma_{n}} g_{a_{n+1}}      
 \nonumber\\   
&=&  
  - \frac{1}{q} g_{\sigma_{n}} g_{a_{n+1}}    
 -  \frac{1}{q}  g_{a_{n+1}}   g_{\sigma_{n}} - \frac{1}{q} g_{\sigma_{n}} -  \frac{1}{q}  g_{a_{n+1}}  
 -\frac{1}{q} g_1 
 + (\frac{1}{q}-1) g_{\sigma_{n}} g_{a_{n+1}}      
 \nonumber\\   
&=&  
  - g_{\sigma_{n}} g_{a_{n+1}}    
 -  \frac{1}{q}  g_{a_{n+1}}   g_{\sigma_{n}} - \frac{1}{q} g_{\sigma_{n}} -  \frac{1}{q}  g_{a_{n+1}}  
 -\frac{1}{q} g_1 
  . 
  \nonumber 	 
			\end{eqnarray}
Since $R_n$ is a  homomorphism of algebras we have  
$ R_n (t_w)  =  R_n(t_u)  R_n(t_{a_n})  R_n(t_v)  	 $ and, by  Proposition~\ref{morphismGn}, 
we have   $R_n(t_v)= g_{I(v)}= g_{J(v)}$. Furthermore,  multiplying   on the left 
the element $v\in  W^c( A_{n-1})$   
by any element of $\{ \sigma_{n }a_{n+1}, 
 {a_{n+1}}    {\sigma_{n}}, {\sigma_{n}}  ,  {a_{n+1}}    \}$ 
  produces a reduced fully commutative word, hence:  
\begin{eqnarray}
 R_n(t_{a_n})  R_n(t_v) &=&    
  - g_{\sigma_{n} a_{n+1}I(v)}    
 -  \frac{1}{q}  g_{a_{n+1} \sigma_{n}I(v)} - \frac{1}{q} g_{\sigma_{n}I(v)} -  \frac{1}{q}  g_{a_{n+1}I(v)}  
 -\frac{1}{q} g_{I(v)} . 
 \nonumber 
			\end{eqnarray}
Finally, since $u$ also belongs to $W^c( A_{n-1})$ 
we get, as claimed: 
\begin{eqnarray}
 R_n (t_w)  &=&    g_{I(u)}     \ 
 \left( - g_{\sigma_{n} a_{n+1}I(v)}    
 -  \frac{1}{q}  g_{a_{n+1} \sigma_{n}I(v)} - \frac{1}{q} g_{\sigma_{n}I(v)} -  \frac{1}{q}  g_{a_{n+1}I(v)}  
 -\frac{1}{q} g_{I(v)} 
   \right) 
 \nonumber \\
 &=& -g_{I(w)} -\frac{1}{q} g_{J(w)}  -\frac{1}{q} \sum_{\begin{smallmatrix} l(x)< l(I(w)) \cr L(x) \le 1 \end{smallmatrix}} \alpha_x g_x  \qquad (\alpha_x \in K). \nonumber 
			\end{eqnarray}

We now assume that the property holds for any $u$ of  positive affine length at most $k$. 
Any  $w \in  W^c(\tilde A_{n-1})$ with $L(w)= k+1$ can be written as 
$w = u a_n v$ where $u \in  W^c(\tilde A_{n-1})$,  $L(u)= k $,  
$v \in W^c( A_{n-1})$, and $l(w) = l(u)+ l( v) +1$. We have $ R_n (t_w)  =  R_n(t_u)  R_n(t_{a_n})  R_n(t_v)  	 $. Using our previous computation of    $ R_n(t_{a_n})  R_n(t_v)$ and the induction hypothesis we write: 
$$ \begin{aligned}
 R_n (t_w)   	 
  &=   \big((-1)^{L(u)}  g_{I(u)} + (-\frac{1}{q})^{L(u)}   g_{J(u)}  + \sum_{\begin{smallmatrix} l(x)< l(I(u)) \cr L(x) \le L(u) \end{smallmatrix}}  \alpha_x  g_x  \big)  \\
&\qquad \qquad \qquad 
 \big( - g_{\sigma_{n} a_{n+1}I(v)}    
 -  \frac{1}{q}  g_{a_{n+1} \sigma_{n}I(v)} - \frac{1}{q} g_{\sigma_{n}I(v)} -  \frac{1}{q}  g_{a_{n+1}I(v)}  
 -\frac{1}{q} g_{I(v)}
   \big). \end{aligned}
 $$
We know from  Theorem~\ref{IJ} that $I(w)= I(u) \sigma_n a_{n+1} I(v)$ and $J(w)= J(u) a_{n+1}\sigma_n  I(v)$
and that both are reduced fully commutative words, hence 
$g_{I(u)} g_{\sigma_{n}a_{n+1}I(v)} = g_{I(w)}$ and  
$g_{J(u)}g_{a_{n+1}  \sigma_{n}I(v)} = g_{J(w)}$. We obtain the two leading terms 
$$ (-1)^{L(w)}   g_{I(w)} +  (-\frac{1}{q})^{L(w)}   g_{J(w)} $$ 
 in the formula that we are looking for. \\

We now observe  that   the other terms in the development of the product above have affine length at most 
$L(w)$ and Coxeter length  at most 
$l(I(w))$. The  only terms that might have      length $l(I(w))$  come from $I(u)$ and  $J(u) $  in the first parenthesis, 
together with 	$\sigma_n a_{n+1} I(v)$ and   $a_{n+1}  \sigma_{n} I(v)$ 
   in the second.  The cases of $I(w)$ and $J(w)$ being settled, it remains  to prove  that $g_{I(u)} g_{a_{n+1}  \sigma_{n}I(v)}$ and 
$g_{J(u)} g_{\sigma_{n}a_{n+1}I(v)}$  are linear combinations of basis elements $g_x$ where the length of 
$x \in  W^c(\tilde A_{n})$  is  strictly less than $l(I(w))$.

Remember from Lemma~\ref{lemmafull} that between two consecutive appearances 
of $a_{n+1}$ we must see one and only one  occurrence of $\sigma_n$. So the word  
$I(u) a_{n+1}  \sigma_{n} I(v)$    either is not reduced,  hence of length strictly 
  less than $l(I(w))$, or is not fully commutative. In the latter case  it contains 
  a braid   so the corresponding product $g_{I(u)} g_{ a_{n+1}  \sigma_{n} I(v)}$ decomposes, in the Temperley Lieb algebra, 
 into a linear combination of elements $g_z$ with $l(z) < l(I(w))$. 
 Similarly  
 $J(u) \sigma_{n} a_{n+1}   I(v)$   has two occurrences of  $ \sigma_{n}$ 
 between the rightmost  occurrence of $a_{n+1}$ and the previous one on the left: it cannot be fully commutative reduced hence the product $g_{J(u)} g_{\sigma_{n}a_{n+1}I(v)}$  decomposes as before into terms of strictly smaller length.  
 The result follows. 
\end{proof} 

\medskip 

\begin{theorem}\label{F} 
The tower of affine Temperley-Lieb  algebras

\begin{eqnarray}
			\widehat{TL}_{1}(q) \stackrel{R_{1}}{\longrightarrow}  \widehat{TL}_{2}(q) \stackrel{R_{2}} {\longrightarrow}\widehat{TL}_{3}(q) \longrightarrow \dotsb  \longrightarrow \widehat{TL}_{n}(q) \stackrel{R_{n}} {\longrightarrow}\widehat{TL}_{n+1}(q) \longrightarrow \dotsb  \nonumber\\\nonumber
		\end{eqnarray}

	is a tower of faithful arrows. 	\\  
\end{theorem}

\begin{proof} We need to show that $R_n$ is an injective homomorphism of algebras. 
A basis for $\widehat{TL}_{n}(q)$ is given by the elements $t_w$ where $w$ runs over 
$ W^c(\tilde A_{n-1})$. Assume that  there are non trivial dependence relations between the images of these basis elements. Pick one such relation, say 
$\sum_w \lambda_w R_n(t_w) =0 $, and let 
$k = \max \{ l(w) + L(w) \ | \  \lambda_w \ne 0\} $. 
Using Proposition~\ref{formula} we can write  this relation as follows:

$$  \sum_{\begin{smallmatrix} l(w) + L(w) = k  \\  L(w)>0
\end{smallmatrix} } \lambda_w (  (-1)^{L(w)}   g_{I(w)} +  (-\frac{1}{q})^{L(w)}   g_{J(w)} )
 +  \sum_{\begin{smallmatrix}  l(w)+ L(w) = k  \\  L(w)=0\end{smallmatrix} } \lambda_w  g_{I(w)} + 
 \sum_{l(x)<k}  \lambda'_x g_x  = 0 $$
 
 \medskip\noindent  
 for suitable coefficients $\lambda'_x$ (where the $x$'s are elements of  
 $W^c(\tilde A_{n}) $ and $l(x)$ is the length in $W (\tilde A_{n}) $). 
   Since the elements $g_y$ for 
 $y  \in  W^c(\tilde A_{n})$ form a basis of $\widehat{TL}_{n+1}(q)$, 
 and since $I$ and $J$ are injective and the intersection of their images  is 
 $   W^c(  A_{n-1})$, we see that all the coefficients $\lambda_w$ for 
 $l(w)+L(w)=k$ must be $0$, a contradiction. 
\end{proof}

	\medskip

\renewcommand{\refname}{REFERENCES}

\end{document}